\documentclass[a4paper, 12pt, notitlepage]{amsart}
\usepackage[colorlinks]{hyperref} 
\usepackage{latexsym, amssymb, amscd, amsfonts, amsthm, amsmath, graphicx,
mathabx, mathrsfs, tikz}

\newcommand{\cthree}{\operatorname{3-cycle}}
\newcommand{\srw}{\operatorname{srw}}
\newcommand{\ws}{\operatorname{ws}}
\newcommand{\nws}{\operatorname{nws}}
\newcommand{\Pois}{\operatorname{Pois}}

\newcommand{\supp}{\operatorname{supp}}

\newcommand{\bad}{\operatorname{bad}}
\newcommand{\tr}{\operatorname{tr}}

\newcommand{\TV}{\operatorname{TV}}

\newcommand{\Alt}{\operatorname{Alt}}
\newcommand{\Brownian}{\operatorname{Br}}
\newcommand{\puzzle}{\operatorname{puz}}

\newcommand{\bN}{\mathbb{N}}
\newcommand{\bR}{\mathbb{R}}

\newcommand{\bU}{\mathbb{U}}
\newcommand{\zed}{\mathbb{Z}}

\newcommand{\id}{\mathrm{id}}

\newcommand{\one}{\mathbf{1}}
\newcommand{\E}{\mathbf{E}}
\newcommand{\Var}{\mathbf{Var}}
\newcommand{\Prob}{\mathbf{Prob}}

\newcommand{\sD}{\mathscr{D}}
\newcommand{\sE}{\mathscr{E}}

\newcommand{\sR}{\mathscr{R}}

\newcommand{\sX}{\mathscr{X}}

\newcommand{\ua}{\underline{a}}

\newcommand{\sP}{{\mathscr{P}}}

\newtheorem{theorem}{Theorem}

\newtheorem{cor}[theorem]{Corollary}
\newtheorem{lemma}[theorem]{Lemma}
\newtheorem*{lemma*}{Lemma}
\newtheorem{proposition}[theorem]{Proposition}

\theoremstyle{remark}
\newtheorem*{rem}{Remark}
\newtheorem{definition}{Definition}

\title{Solution of the 15 puzzle problem}
\author{Yang Chu}
\address{Department of Mathematics, Stony Brook University, Stony Brook,
NY, 11794}
\email[Yang Chu]{yang.chu@stonybrook.edu}

\author{Robert Hough}
\email[Robert Hough]{robert.hough@stonybrook.edu}


\thanks{This material is based upon work supported by the National Science
Foundation under agreement No.\ DMS-1802336. Any opinions, findings and
conclusions or recommendations expressed in this material are those of the
author and do not necessarily reflect the views of the National Science
Foundation.}

\thanks{Yang Chu was supported by a fellowship from the Summer Math Foundation at Stony Brook.}

\begin{document}

\begin{abstract}
A generalized `$15$ puzzle' consists of an $n \times n$ numbered grid, with one missing number.  A move in the game switches the position of the empty square with the position of one of its neighbors.  We solve  Diaconis' `15 puzzle problem' by proving that the asymptotic total variation mixing time of the board is at least order $ n^4 $ when the board is given periodic boundary conditions and when random moves are made.  We  demonstrate  that for any $f(n) \to \infty$ with $n$, the number of fixed points after $n^4 f(n)$ moves converges to a Poisson distribution of parameter 1.  The order of total variation mixing time for this convergence is $n^4$ without cut-off.  We also prove an upper bound of order $n^{4 }\log n$ for the total variation mixing time. 
\end{abstract}

\maketitle

\section{Introduction}

A `15 puzzle' consists of a $4 \times 4$ board with 15 numbered unit tiles and one empty square.  A move in the puzzle consists of sliding a numbered tile into the empty square.  The 15 puzzle gained notoriety in the United States in the 1870's when an article in the American Journal of Math \cite{JS79} asked whether the board with positions 14 and 15 exchanged and an empty tile in the lower right corner could be shifted into sorted order, again with the empty tile in its initial position, see Figure \ref{fig:15_puzzle} (it cannot, the group of permutations generated is $A_{15}$). In general, an $n^2-1$ puzzle consists of an $n \times n$ board with $n^2-1$ numbered tiles and one empty square.   In the book \cite{D88}, Diaconis considers the problem of randomizing an $n^2-1$ puzzle given periodic boundary conditions by, at each step, shifting a uniform random neighbor of the open square into its place and makes conjectures for the mixing time of a single labeled piece on the board and for the entire puzzle.  Our main results solve Diaconis' `15 puzzle' problem.
\begin{figure}
 \centering
\begin{tikzpicture}[node distance = 1.3cm, auto, place/.style = {circle,  thick, draw=blue!75, fill=blue!20}]
 \draw (-2, -2)--(2,-2);
 \draw (-2,-1)--(2,-1);
 \draw (-2,0)--(2,0);
 \draw (-2,1)--(2,1);
 \draw (-2,2)--(2,2);
 \draw(-2,-2)--(-2,2);
 \draw(-1,-2)--(-1,2);
 \draw(0,-2)--(0,2);
 \draw(1,-2)--(1,2);
 \draw(2,-2)--(2,2);
 \draw (-1.5, 1.5) node {1};
 \draw (-.5, 1.5) node {2};
 \draw (.5, 1.5) node {3};
 \draw (1.5, 1.5) node {4};
 \draw (-1.5, .5) node {5};
 \draw (-.5, .5) node {6};
 \draw (.5, .5) node {7};
 \draw (1.5, .5) node {8};
 \draw (-1.5, -.5) node {9};
 \draw (-.5, -.5) node {10};
 \draw (.5, -.5) node {11};
 \draw (1.5, -.5) node {12};
 \draw (-1.5, -1.5) node {13};
 \draw (-.5, -1.5) node {15};
 \draw (.5, -1.5) node {14};
 \filldraw[fill=black!10] (1, -1) rectangle (2,-2);  
 \draw[ ultra thick](2,-2)--(2,2);
 \draw[ultra thick] (-2,-2)--(2,-2);
 \draw[ ultra thick] (-2,2)--(2,2);
 \draw[ ultra thick](-2,-2)--(-2,2);
\end{tikzpicture}
\caption{A 15 puzzle. A move in the puzzle slides a numbered tile into the empty space.}\label{fig:15_puzzle}
\end{figure}

\begin{theorem}\label{single_piece_mixing_theorem}
 Let $d_{\Brownian}(t)$ be the total variation distance to uniformity at time $t>0$  of standard Brownian motion started from $(0,0)$ on $(\bR/\zed)^2$.  Let $c_{\puzzle} = \frac{5}{2}(\pi-1)$. As $n \to \infty$, the total variation distance to uniformity of a single piece in the $n^2-1$ puzzle  at time $c_{\puzzle}  n^4 t$ converges to $d_{\Brownian}(t)$ uniformly for $t$ in compact subsets of $(0,\infty)$.
\end{theorem}
Although Theorem \ref{single_piece_mixing_theorem} is stated for convergence in the total variation metric, the order $n^4$ mixing time also holds in the stronger $\frac{\epsilon}{|G|}-\ell^\infty$ metric, see Section \ref{background_section} for a discussion of these metrics.  This stronger convergence is illustrated in the following theorem.  As is well known, the number of fixed points of a uniform random permutation on the symmetric group $S_n$ converges to a $\Pois(1)$ distribution as $n \to \infty$.
\begin{theorem}\label{poisson_theorem}
 Let $f:\bN \to \bN$ be an arbitrary growth function such that $f(n) \to \infty$ as $n \to \infty$.  If an $n^2-1$ puzzle is sampled after $n^4f(n)$ random steps, then the number of pieces in the puzzle in their original position converges to a $\Pois(1)$ distribution.  The convergence does not hold if $f$ remains bounded. 
\end{theorem}

The cut-off phenomenon  for a sequence of Markov chains describes a transition from non-uniform to uniform during a number of steps which is lower order than the mixing time.  Cut-off phenomena have been proved to occur for a wide class of chains, including nearest neighbor random walk on the hypercube and the random transposition walk on the symmetric group. 
As a consequence of the proof of Theorem \ref{poisson_theorem} we obtain the following Corollary.
\begin{cor}\label{no_cut-off_cor}
 The convergence of the number of fixed points in an $n^2-1$ puzzle does not exhibit a cut-off phenomenon in total variation.
\end{cor}

Modifying an argument of \cite{HSZ15} we demonstrate the following mixing time upper bound for the full $n^2-1$ puzzle.
\begin{theorem}\label{upper_bound_theorem}
 The total variation and $\frac{\epsilon}{|G|}-\ell^\infty$ mixing time of an $n^2-1$ puzzle is $O\left(n^4 \log n\right)$.
\end{theorem}
We also give a coupling of the process which tracks the empty square and boundedly many numbered pieces of the $n^2-1$ puzzle.  
\begin{theorem}\label{coupling_theorem}
 For each fixed $d$, as $n \to \infty$, there is a coupling of the Markov process described by the empty square and any $d$ labeled pieces, such that the expected time for the chain started from a deterministic position to coincide with the chain started from stationarity is $ O(n^4 \log n)$.
\end{theorem}


\subsection{Discussion of method}
Theorems \ref{single_piece_mixing_theorem} and \ref{poisson_theorem} are proved by tracking one or several marked pieces on the board as they move.  The pieces move at the times of a renewal process when the empty square moves next to one of them and then the piece is shifted into it.  To prove Theorem \ref{single_piece_mixing_theorem}, a local limit theorem is proved which demonstrates the approximate independence of the piece's location and the number of moves of the empty square after the marked piece has moved approximately $n^2$ times.  Since the expected time of a renewal is of order $n^2$, this takes time order $n^4$.  The lower order fluctuations from the sum of the renewal process times are then absorbed as an error term.  

Theorem \ref{poisson_theorem} is proved by making a comparison to the exclusion process of several labeled pieces on the board, making moves that are again taken at the times of a renewal process. See \cite{DS93b} for an analysis of the spectrum of the exclusion process on a graph with unlabeled pieces.  A slightly different approach is taken here to analyze the exclusion process, passing to a larger state space without exclusion by extending the test functions by averaging.  The convergence to $\Pois(1)$ is essentially equivalent to proving the mixing in each representation of the symmetric group corresponding to a partition with boundedly many pieces below the first row or  with boundedly many pieces to the right of its first column, see the discussion in \cite{D88}.

The proof of the upper bound in Theorem \ref{upper_bound_theorem} works as follows. Let $U$ denote the move that shifts the blank piece up, and $R$ the move that shifts the blank piece right.  The commutator $URU^{-1}R^{-1}$ is a 3-cycle. In order $n$ moves, any three pieces can be moved into the positions adjacent to the empty square in the lower right corner.  Conjugating the 3 cycle by this permutation obtains any 3 cycle in the conjugacy class.  The mixing time upper bound now follows from the comparison bound \cite{DS93b} for the mixing of 3 cycles on the alternating group.  
\subsection{Related work}
Much of the prior work regarding $n^2-1$ puzzles has focused on sorting strategies for a puzzle in general position, and on the hardness of finding shortest sorting algorithms.  For instance, it is known that any $n^2-1$ puzzle may be returned to sorted order in order $n^3$ steps (and fewer are not always possible) \cite{P95}, and that finding the shortest solution is NP-hard \cite{RW86}, \cite{G11}. 

There has been significant interest recently in studying random walks with small generating sets in non-commutative groups, see for instance \cite{BG08a}, \cite{BG08b} and \cite{H08}, with many subsequent works, but there are still comparatively few such walks which have been studied on the symmetric group.  In \cite{DS93a} Diaconis and Saloff-Coste obtain a mixing time upper bound of order $n^3$ for symmetric random walk on the symmetric group generated by a two cycle and an $n$ cycle.  For an arbitrary pair of generators, the best known general upper bound has been obtained by Helfgott and Seress \cite{HS14}, but is not polynomial in $n$.  For a uniformly chosen pair of random generators, their bound is $O(n^3 (\log n)^c)$.  

A host of random walks on permutation groups  have been studied via a variety of techniques, including the representation theory \cite{DS81}, \cite{H16}, \cite{BN19}, and couplings \cite{BS19}. Our method for the upper bound, which uses a three transitive group action, is based on \cite{HSZ15}.  So far as we are aware, the method of using renewal theory together with a local limit theorem to prove the lower bound, is new.  The comparison techniques used in the proof of Theorem \ref{poisson_theorem} are partly based on \cite{DS93b}.

\section{Background and notation}\label{background_section}
The following asymptotic notation is used.  Unless otherwise stated, all asymptotics are taken as the side length $n$ of the puzzle tends to infinity.  The notation $A\ll B$ has the same meaning as $A = O(B)$, and means that there is a fixed constant $C$ such that $|A| \leq C B$. The notation $A=O_D(B)$ means that the constant $C$ is allowed to depend on a further parameter $D$. The notation $A \asymp B$ means $A\ll B$ and $B \ll A$.  The notation $A \sim B$ means $A = (1 + o(1))B$ or $\lim \frac{A}{B} = 1$.

For the reader's convenience we review the basic conventions regarding Markov chains, see \cite{LPW09}. Denote by $\sX$ a finite or countable state space, and by $P$ a Markov matrix.  Given two states $x, y \in \sX$, $P(x,y) = e_x^t P e_y$ denotes the probability of transferring from state $x$ to state $y$.  Here $e_x$ and $e_y$ are the standard basis vectors.  A Markov chain $X_0, X_1, X_2, ...$ is a collection of $\sX$ valued random variables satisfying the Markov property
\begin{align}
 &\Prob\left(X_{t+1} = y \Big| \bigcap_{s = 0}^{t-1} \{X_s = x_s\} \cap \{X_t = x\}\right) \\\notag&= \Prob\left(X_{t+1} = y| X_t = x\right) = P(x,y).
\end{align}
Thus $P^t$ denotes the Markov transition matrix of $t$ steps of the chain. The chain is irreducible if, given any two states $x, y$, there is a positive probability of moving from $x$ to $y$ in finitely many steps.  
A finite space irreducible Markov chain has a unique stationary probability distribution $\pi$ which satisfies $\pi^t P = \pi^t$. Denote $\Prob_x$ and $\E_x$ probability and expectation for a chain started from the state $X_0 = x$.  

The hitting time of a state $x \in \sX$ is $\tau_x := \min\{t \geq 0: X_t = x\}$.  Let $\tau_x^+ = \min\{t \geq 1: X_t = x\}$.  When $X_0 = x$, $\tau_x^+$ is called the first return time.  The unique stationary distribution of a finite state irreducible Markov chain satisfies 
\begin{equation}
 \pi(x) = \frac{1}{\E_x[\tau_x^+]}.
\end{equation}

A stopping time $\tau$ for $(X_t)$ is a $\{0, 1, 2, ...\}\cup \{\infty\}$ valued random variable satisfying $\{\tau = t\}$ is determined by $X_0, ..., X_t$. All of the stopping times considered here will be almost surely finite.  An important property in this work is the \emph{strong Markov property}
\begin{align}
 &\Prob_{x_0}\{(X_{\tau+1}, X_{\tau+2}, ..., X_{\tau+\ell} \in A| \tau = k \wedge (X_1, ..., X_k) = (x_1, ..., x_k)\}\\
 \notag&= \Prob_{x_k}\{(X_1, ..., X_\ell) \in A\}.
\end{align}

Given a function $f: \sX \to \bR$, the action of $P$ on $f$ is defined by
\begin{equation}
 Pf(x) = \sum_{y \in \sX} P(x,y)f(y).
\end{equation}
The function $f$ is said to be harmonic at $x$ if $Pf(x) = f(x)$.  The Markov chain $P$ is called reversible if, for all $x, y$, \begin{equation}\pi(x) P(x,y) = \pi(y) P(y,x).\end{equation}  All chains in this work will be irreducible and reversible, with stationary distribution that is uniform on $\sX$.  Under this condition, the matrix $P$ is symmetric with real eigenvalues $1 = \lambda_0 > \lambda_1 \geq \cdots \geq \lambda_{|\sX|-1} \geq -1$.

The time $t$ heat kernel associated to $P$ is
\begin{equation}
 H_t(P) = e^{-t}\sum_{k=0}^\infty \frac{t^k P^k}{k!}. 
\end{equation}
Write $\sigma(P)$ for the spectrum, including multiplicity, of $P$.
If $P = \sum_{\lambda \in \sigma(P)} \lambda v_\lambda v_\lambda^t$ is a diagonalization of $P$ in an orthonormal eigenbasis $\{v_\lambda\}_{\lambda \in \sigma(P)}$ then 
\begin{equation}
 H_t(P) = \sum_{\lambda \in \sigma(P)} e^{(\lambda-1)t} v_\lambda v_\lambda^t.
\end{equation}
Given a smooth function $\phi \geq 0$ of compact support, $\int_{\bR^+} \phi = 1$ on $\bR^+$, define the Laplace transform 
\begin{equation}
 \Phi_t(P) = \int_0^\infty \phi(s)H_{st}(P) ds = \sum_{\lambda \in \sigma(P)} \hat{\phi}((1-\lambda)t) v_\lambda v_\lambda^t
\end{equation}
where
\begin{equation}
 \hat{\phi}(t) = \int_0^\infty \phi(s)e^{-st}ds
\end{equation}
is the usual Laplace transform. The relation between the spectrum of the original chain $P$ and the Laplace transform $\Phi_t(P)$ will be useful in proving Theorem \ref{poisson_theorem}.
The Dirichlet form associated to $P$ is a quadratic form
\begin{equation}
 \sE(f,f) = \langle (I-P)f,f\rangle = \frac{1}{2} \sum_{x,y} (f(x)-f(y))^2 \pi(x) P(x,y).
\end{equation}

A random walk on a finite group $G$ with driving probability measure $\mu$ can be interpretted as a Markov chain in which $\sX = G$ and $P(x,y) = \mu(x^{-1}y)$.  The distribution after $n$ steps of the random walk started from the identity is given by the group convolution
\begin{equation}
 \mu^{*n} (x) = \mu*\mu^{*(n-1)}( x)
\end{equation}
where
\begin{equation}
 \mu * \nu(z) = \sum_{xy=z} \mu(x)\nu(y).
\end{equation}
If the generating measure $\mu$ is symmetric, that is $\mu(x) = \mu(x^{-1})$ for all $x$, and if it is lazy, that is $\mu(0) = \epsilon > 0$, then the stationary measure is uniform on $G$, the Markov chain is reversible, and $\mu^{*n} \to \bU_G$ as $n \to \infty$.  Also, the least eigenvalue is bounded below by $-1+\epsilon$.

Several standard metrics are used on finite state Markov chains.  The total variation metric between two probability measures $\mu$, $\nu$ on $\sX$ is
\begin{equation}
 \|\mu-\nu\|_{\TV} = \sup_{A \subset \sX} |\mu(A) - \nu(A)| = \frac{1}{2} \sum_{x \in \sX} |\mu(x)-\nu(x)|.
\end{equation}
The $d_2$ distance is a scaled version of the $\ell^2$ norm,
\begin{equation}
 \|\mu- \nu\|_{d_2}^2 = |\sX| \sum_{x \in \sX} (\mu(x)-\nu(x))^2.
\end{equation}
For $0 < \epsilon < 1$, the $\frac{\epsilon}{|\sX|}$-$\ell^\infty$ distance between $\mu$ and $\nu$ is
\begin{equation}
 \|\mu-\nu\|_{\epsilon, \infty} = \frac{|\sX|}{\epsilon} \sup_{x \in \sX} |\mu(x)-\nu(x)|.
\end{equation}
Given any of these metrics, the mixing time of the chain $\{e_{x}^tP^N\}$ to uniformity $\nu$ is the first steps $N$ such that $\|e_{x}^t P^N - \nu\|<\frac{1}{e}$. For symmetric random walk on a group, the total variation mixing time and $\frac{\epsilon}{|G|}-\ell^\infty$ mixing time are bounded up to constants by the $d_2$ mixing time, see \cite{DS93a}.

We make the following conventions regarding probability.  The standard one dimensional Gaussian distribution is
\begin{equation}
 \eta(x) = \frac{1}{\sqrt{2\pi}} e^{-\frac{x^2}{2}}.
\end{equation}
A $d$ dimensional Gaussian of mean $\mu$ and covariance matrix $\sigma^2$ is
\begin{equation}
 \eta_{\mu, \sigma}(x) = \frac{1}{(2\pi)^{\frac{d}{2}} \det\sigma} \exp\left(-\frac{1}{2} (x-\mu)^t \sigma^{-2}(x-\mu)\right).
\end{equation}
Normalize two dimensional Brownian motion $B(t)$ started from $0$ in $\bR^2$ by giving $B(1)$ a Gaussian distribution of mean 0 and covariance $I_2$, the two dimensional identity matrix. Brownian motion on $\bR^2/\zed^2$ is obtained by projection to the quotient. 

We use repeatedly that sums of independent random variables and related quantities are concentrated, in the sense that their tail distribution decays more than polynomially away from the mean.  Specifically, a sequence of events $\{E_n\}$ on a sequence of probability spaces $\{\sX_n\}$ is said to hold with overwhelming probability (w.o.p.) if, for any constant $A>0$,
\begin{equation}
 \Prob(E_n) = 1 - O_A(n^{-A}),
\end{equation}
that is, the probability of the complement of the events decays faster than any polynomial in $n$.

We use the following conventions for Fourier analysis.  The additive character on $\bR$ is $e(x) = e^{2\pi i x}$ and write $c(x) = \cos(2\pi x)$, $s(x) = \sin(2\pi x)$.  The Fourier transform of $f \in L^1(\bR^d)$ is
\begin{equation}
 \hat{f}(\xi) = \int_{\bR^d} e(\xi\cdot x) f(x)dx.
\end{equation}
If $g(x) = f(x-a)$ then $\hat{g}(\xi) = \hat{f}(\xi)e(\xi\cdot a)$.
If $f$ is Schwarz class then so is $\hat{f}$, and 
\begin{equation}
 f(x) = \int_{\bR^d} e(-\xi \cdot x) \hat{f}(\xi) d\xi.
\end{equation}
The characteristic function, or Fourier transform, of the standard one dimensional Gaussian is given by 
\begin{equation}
 \hat{\eta}(\xi) = \exp(-2\pi^2 \xi^2).
\end{equation}

The Fourier series of an $L^1$ function on $\bR^d/\zed^d$ is
\begin{align}
 f(x) \sim \sum_{n \in \zed^d} \hat{f}(n) e(n \cdot x)
\end{align}
where
\begin{equation}
\hat{f}(n) = \int_{\bR^d/\zed^d} f(x)e(n\cdot x) dx.
\end{equation}
Brownian motion at time $t$ has a distribution on $(\bR/\zed)^2$ which is a theta function $\theta_t(x)$, which has a Fourier series representation
\begin{equation}\label{theta_fourier}
 \theta_t(x) = \sum_{n \in \zed^2} e^{-2\pi^2 t \|n\|_2^2} e(n\cdot x).
\end{equation}
For $t$ in compact subsets of $\bR^+$ this function is uniformly $C^j$ for every $j$, as may be checked by differentiating term-by-term.

We use as an explicit function on $\bR/\zed$ the Dirichlet kernel $D_N$ which satisfies 
\begin{equation}
 \hat{D}_N(n) = \frac{\one(|n| \leq N)}{2N+1} .
\end{equation}
This is given by $D_N(0)=1$ and, for $0 \neq x \in \bR/\zed$, 
\begin{equation}
D_N(x) = \left(\frac{1}{2N+1} \frac{s((N+\frac{1}{2})x)}{s(\frac{x}{2})} \right).
\end{equation}
\begin{lemma}\label{dirichlet_lemma}
 The Dirichlet kernel satisfies the following bounds uniformly in $N$.
 \begin{enumerate}
  \item There is a constant $c>0$ such that, for $|x| \leq \frac{1}{2N}$,
 $|D_N(x)| \leq 1-cN^2x^2.$
 \item For each $\epsilon > 0$ there is a constant $c(\epsilon)>0$ such that, for $\frac{\epsilon}{N} \leq |x| \leq \frac{1}{2}$, $|D_N(x)| \leq 1- c(\epsilon).$
 \end{enumerate}
\end{lemma}
\begin{proof}
 Write 
 \begin{equation}
  D_N(x) = \frac{1}{2N+1}\left(1 + 2\sum_{j=1}^N c(jx) \right) = 1 - \frac{2}{2N+1} \sum_{j=1}^N (1-c(jx)).
 \end{equation}
To prove the first claim, use the bound, uniformly in $|x| \leq \frac{1}{2}$, $1 - c(x) \gg x^2$.

To prove the second claim, it suffices to consider the range $\frac{\epsilon}{N} \leq |x| \leq \frac{C}{N}$ for a fixed constant $C$, since for larger $|x|$ the claim follows on bounding $\left|s\left(\frac{x}{2}\right)\right|\gg |x|$ and $|s((n+\frac{1}{2})x)|\leq 1$.  In the stated range, a positive fraction of the terms in the sum $\sum_{j=1}^N 1-c(jx)$ are bounded below by a fixed constant depending only on $\epsilon$, which proves the claim by again using $1-c(x) \gg x^2$.
\end{proof}

For a finite abelian group $G$, the Fourier transform at a character $\chi$ of function $f$ is
\begin{equation}
 \hat{f}(\chi) = \sum_{g \in G} f(g)\chi(g).
\end{equation}
The inversion formula is given by
\begin{equation}
 f(g) = \frac{1}{|G|} \sum_{\chi \in \hat{G}} \hat{f}(\chi)\overline{\chi}(g).
\end{equation}
This is a special case of the Fourier inversion formula for a function on an arbitrary finite group, which is given by (see e.g. \cite{D88})
\begin{align}
f(g)&= \frac{1}{|G|}\sum_{\rho \text{ irrep.}}d_\rho \tr\left[ \rho\left(g^{-1}\right)\hat{f}(\rho)\right]\\ \notag
\hat{f}(\rho)& = \sum_{h \in G} f(h) \rho(h).
\end{align}
This Fourier transform has the usual property of carrying group convolution 
to pointwise product,
\begin{equation}
 \widehat{\mu * \nu}(\rho) = \hat{\mu}(\rho) \hat{\nu}(\rho).
\end{equation}

We write $\sD_M$ for the distribution of $M$ steps of simple random walk on $\zed^2$, which has a multinomial distribution.

\subsection{Specific chains}\label{list_of_chains}
The following Markov chains appear in this work.
\begin{itemize}
 \item $(\sX_{n^2-1}, P_{n^2-1})$, the full $n^2-1$ puzzle including $n^2-1$ labeled pieces and the empty square.  This Markov chain is shown to be a symmetric random walk on a group in Section \ref{random_walk_section}, and hence is reversible with uniform stationary measure on the group.
 \item $(\sX_{d}, P_d)$.  This chain is obtained from the $n^2-1$ puzzle by forgetting the location of all but $d$ labeled pieces and the empty square. It inherits the property of reversibility and uniform stationary measure from the $n^2-1$ puzzle.
 \item $(\sX, P)$.  This is the case $d = 1$ of a single labeled piece, $(\sX_1, P_1)$.
 \item $(\sX_{d,s}, P_{d,s})$.  This is a symmetrized version of $(\sX_d, P_d)$, which is described as follows.  The state space $\sX_{d,s}\supset \sX_d$ is enlarged by allowing the empty square $\sP_e$  to take any position on $(\zed/n\zed)^2$ including those states occupied by some labeled piece $\sP_i$.  In addition, at most two of the $\sP_i$ may overlap if the empty square overlaps them as well.  The transitions are as follows.  If $\sP_e$ does not overlap any $\sP_i$ then it transitions to one of its neighbors with probability $\frac{1}{10}$ and holds with probability $\frac{3}{5}$.  If $\sP_e$ overlaps a single $\sP_i$ then either $\sP_i$ and $\sP_e$ move to a neighbor each with probability $\frac{1}{10}$, $\sP_e$ moves alone to any neighbor, each with probability $\frac{1}{10}$, or the position holds with the remaining probability. If $\sP_e$ overlaps $\sP_i$ and $\sP_j$, either $\sP_i$ or $\sP_j$ but not both can move with $\sP_e$ to a neighbor each with probability $\frac{1}{10}$, or the position holds with probability $\frac{1}{5}$. 
 
 By construction, $P_{d,s}(x, y) = P_{d,s}(y, x)$ so $P_{d,s}$ is reversible with stationary measure which is  uniform on $\sX_{d,s}$.
 
 \item $(\sX_{d,M}, P_{d,M})$.  This is a further symmetrized chain which depends on a parameter $M$.  The state space $\sX_{d, M} = ((\zed/n\zed)^2)^{(d+1)}$ in which the first $d$ labeled pieces positions can overlap, and the last position in $(\zed/n\zed)^2$ is the empty square, which can also overlap any position. The transition kernel $P_{d,M}$ is a random walk.  With probability $\frac{1}{2}$ the position holds.  With probability $\frac{1}{2}$ each of the first $d$ positions transitions independently with $x$ and $y$ coordinates each transitioning to a uniform point at distance at most $M$ from the current position.    The empty piece independently transitions to a uniform point on $(\zed/n\zed)^2$.
 
 \item $(\sX_{\srw}, P_{\srw})$ is nearest neighbor simple random walk on $(\zed/n\zed)^2$ in which the piece transitions to one of its four neighbors with probability $\frac{1}{4}$ each.  In $(\sX_{n^2-1}, P_{n^2-1})$, $(\sX_d, P_d)$, and $(\sX_{d,s}, P_{d,s})$, considering only the position of the empty piece, and tracking only the times at which it moves obtains $(\sX_{\srw}, P_{\srw})$.
 
 \item $(\sX_{\cthree}, P_{\cthree})$.  This is random walk on the alternating group $A_{n^2-1}$ with driving measure which is uniform on the set of 3-cycles.
\end{itemize}

\section{Description of the random walk}
\label{random_walk_section}
Identify the $n \times n$ board of the $n^2-1$ puzzle with $(\zed/n\zed)^2$ with the pieces in initial position numbered left to right and then top down, with the empty square at $(n,n)$.  
\begin{theorem}
 The $n^2-1$ puzzle Markov chain, in which at each step the empty square moves right, left, up, down, or does not move with equal probability can be identified with random walk on the group
 \begin{equation}
  G = S_{n^2-1} \times (\zed/n\zed)^2
 \end{equation}
driven by the measure
\begin{equation}
 \mu = \frac{1}{5} \left(\delta_{\id} + \delta_{R} + \delta_L + \delta_U + \delta_D\right),
\end{equation}
where
\begin{align}
  R&= \left[\begin{array}{c}(n,n-1, \cdots, 1)\\ (2n, 2n-1,\cdots, n+1) \\ \vdots \\(n^2 -n, n^2-n-1, \cdots, n^2-2n+1)\\ (n^2-1, n^2-2, \cdots, n^2-n+1) \end{array}\right] \times (1,0)\\
 \notag U&= \left[\begin{array}{c}(1, n+1, \cdots, n^2-n+1)\\(2, n+2, \cdots, n^2 -n+2) \\ \vdots \\ (n-1, 2n-1, \cdots, n^2-1) \\ (n, 2n, \cdots, n^2-n) \end{array} \right] \times (0,1)
\end{align}
and $L = R^{-1}$, $D = U^{-1}$.  Each generator is the product of $n-1$ $n$-cycles and one $n-1$ cycle.
When $n$ is even, the random walk is supported on group elements for which the parity of the permutation is equal to the parity of the sum of the coordinates on $(\zed/n\zed)^2$.
\end{theorem}
\begin{proof}
 Each time the empty piece moves, record its position on $(\zed/n\zed)^2$, where it performs simple random walk.  On the board, make a translation of $(\zed/n\zed)^2$ to fix the blank piece's position in the lower right corner, so that the board appears to be as from the empty piece's perspective. This identifies the $n^2-1$ puzzle's position with a labeling of $n^2-1$ positions other than the lower right corner, together with a position in $(\zed/n\zed)^2$.  Each move consists of moving one unit in $(\zed/n\zed)^2$, together with the corresponding permutation of the remaining $n^2-1$ positions, so that the Markov chain can be identified with group convolution, hence is the stated random walk on a group.
 
 Started from the initial position, when the empty square makes one move right on the puzzle board its position is now directly below the column of squares labeled $(1, n+1, 2n+1, ..., n^2-n+1)^t$. Thus, from the empty square's perspective, piece 1 is now in the position of piece $n$, piece $n$ in position $n-1$, etc, and similarly in the $n-1$ rows above the last row. Thus each of these rows performs an $n$-cycle. In the last row, piece $n^2-n+2$ is now directly to the right of the empty square, in the position of $n^2 -n+1$, and the pieces other than $n^2-n+1$ are shifted left, forming an $n-1$ cycle.  
 
 When the empty square moves one unit up on the board started from its initial position, from its perspective the first $n-1$ columns perform an $n$-cycle, and the last column performs an $n-1$ cycle.
\end{proof}

\section{The spectrum and comparison}

Given a symmetric $n \times n$ matrix $P$ with eigenvalues $\lambda_0 \leq \lambda_1 \leq \cdots \leq \lambda_{n-1}$, recall that the minimax characterization of the eigenvalues is as follows.  For a subspace $W$ of $\bR^n$, define
\begin{align}
 m(W) &= \min\{ \langle Pf, f \rangle/\langle f, f \rangle: f \in W \setminus \{0\} \},\\
 \notag M(W) &= \max\{ \langle Pf, f \rangle/\langle f, f \rangle: f \in W \setminus\{0\}\}.
\end{align}
The minimax characterization of eigenvalues gives
\begin{align}
 \lambda_i &= \max\{m(W): \dim(W^\perp) = i\} \\
\notag &= \min\{M(W): \dim(W) = i+1\}.
\end{align}
By further constraining the minimum, one obtains an upper bound, for any subspace $U$, 
\begin{equation}
 \lambda_i \leq \lambda_i' = \max \{m(W \cap U): \dim(W^\perp) = i\}.
\end{equation}

In this paper, comparison of Dirichlet forms is used to compare the eigenvalues of related Markov chains on the same state space, following the method described in \cite{DS93b}.
Given a second Markov chain $\tilde{P}(x,y)$ with stationary measure $\tilde{\pi}$, the minimax characterization of eigenvalues leads to the bounds, for $1 \leq i \leq |X|-1$,
\begin{equation}
 \beta_i \leq 1 - \frac{a}{A} (1-\tilde{\beta}_i), \qquad \text{ if } \tilde{\sE} \leq A\sE, \tilde{\pi} \geq a \pi.
\end{equation}
Note that the bound $\tilde{\sE} \leq A \sE$ holds if, for all $x, y$, $\tilde{P}(x,y) \leq A P(x,y)$.
A more advanced method for estimating $A$ is given in \cite{DS93b} which uses paths.  Given $x, y \in \sX$, $x \neq y$ with $\tilde{P}(x,y)>0$, fix a path $\gamma_{xy}$ with steps $x_0 = x, x_1, x_2, ..., x_k = y$ with $P(x_i, x_{i+1}) > 0$.  Set $E = \{(x,y): P(x,y)>0\}$, $\tilde{E} = \{(x,y): \tilde{P}(x,y)>0\}$ and, for $e \in E$,
\begin{equation}
 \tilde{E}(e) = \{(x,y) \in \tilde{E}: e \in \gamma_{xy}\}.
\end{equation}
\begin{theorem}[\cite{DS93b}, Theorem 2.1]\label{path_theorem}
 Let $\tilde{P}, \tilde{\pi}$ and $P,\pi$ be reversible Markov chains on a finite set $\sX$.  The inequality
 \begin{equation}
  \tilde{\sE} \leq A \sE
 \end{equation}
holds with
\begin{equation}
 A = \max_{(z,w) \in E}\left\{\frac{1}{\pi(z)P(z,w)} \sum_{\tilde{E}(z,w)} |\gamma_{xy}| \tilde{\pi}(x)\tilde{P}(x,y) \right\}.
\end{equation}
\end{theorem}
The theorem above is used in conjunction with the following Theorem regarding the spectrum of the symmetrized chain $(\sX_{d,M}, P_{d,M})$ to prove Theorem \ref{poisson_theorem}.
\begin{theorem}
 Let $M = M(n)$ vary with $n$ in the range $1 \leq M \ll \exp(\sqrt{\log n})$.  For fixed $d$, uniformly in $n>1$, as $c \to \infty$, 
 \begin{equation}\label{spectral_sum_d_M}
  \sum_{1 \neq \lambda \in \sigma(P_{d,M})} \lambda^{c \left( \frac{n}{M}\right)^2} \to 0.
 \end{equation}

\end{theorem}
\begin{proof}
 Since $P_{d,M}$ is a random walk on the abelian group $((\zed/n\zed)^2)^{d+1}$ its transition kernel is diagonalized by the characters of the group.  Since the walk is $\frac{1}{2}$-lazy, the spectrum is non-negative, so it will suffice to give an upper bound for the spectrum.
 
 Index the dual group by $\ua = (a_1, ..., a_{d+1}) \in ((\zed/n\zed)^2)^{d+1}$ and write $\lambda_{\ua}$ for the corresponding eigenvalue.  Since the walk holds with probability $\frac{1}{2}$, and otherwise has the coordinates move independently, $\lambda_{\ua} = \frac{1}{2} + \frac{1}{2}\tilde{\lambda}_{\ua}$ where $\tilde{\lambda}_{\ua} = \prod_{j=1}^{d+1} \tilde{\lambda}_{a_j}$ is the eigenvalue of the individual movements, which factors as a product over the coordinates.  Since the last two coordinates move to uniform, $\tilde{\lambda}_{a_{d+1}} = \delta(a_{d+1} = 0)$.  Since each other coordinate moves a uniform distance on $[-M,M]$, its Fourier transform is given by the Dirichlet kernel, so that, for $1 \leq j \leq d$,
 \begin{equation}
 \tilde{\lambda}_{a_j} = D_M \left(\frac{a_{j,1}}{n} \right)D_M \left(\frac{a_{j,2}}{n} \right).
 \end{equation}
There are $n^{2d}(n^2-1)$ terms with $a_{d+1} \neq 0$, which contribute to (\ref{spectral_sum_d_M}) a quantity
\begin{equation}
 n^{2d}(n^2-1) \left(\frac{1}{2} \right)^{c \left(\frac{n}{M} \right)^2}
\end{equation}
which tends to 0 as $c \to \infty$ uniformly in $n$.

By Lemma \ref{dirichlet_lemma}, for each $\epsilon > 0$ there is a $c(\epsilon)$ such that if $\left|\frac{a_{j,i}}{n}\right| \geq \frac{\epsilon}{M}$ then $\left|D_M\left(\frac{a_{j,i}}{n}\right)\right| \leq 1-c(\epsilon)$.  Since $\tilde{\lambda}$ is bounded away from 1 by a fixed constant for these terms, they may be bounded as for the case $a_{d+1} \neq 0$, and are negligible.  For the remaining terms, $D_M\left(\frac{a_{j,i}}{n}\right) \leq 1 - C \frac{M^2 a_{j,i}^2}{n^2}$.  Using the bound $1-x \leq e^{-x}$, which is valid for $0 \leq x \leq 1$, it follows that when $a_{d+1} = 0$, 
\begin{equation}
\lambda_{\ua} \leq e^{-C \frac{M^2 \|\ua\|_2^2}{n^2}}.
 \end{equation}
 Thus 
 \begin{equation}
  \sum_{\ua \neq 0, a_{d+1}=0} \lambda_{\ua}^{c\left(\frac{n}{M} \right)^2} \leq o(1) + \sum_{\ua \in \zed^{2d}, \ua \neq 0} e^{-cC \|\ua\|_2^2}
 \end{equation}
where $o(1)$ tends to 0 as $c \to \infty$.

Since 
\begin{equation}
 \sum_{\ua \in \zed^{2d}, \ua \neq 0} e^{-cC \|\ua\|_2^2} = -1 + \left(\sum_{n \in \zed} e^{-cCn^2} \right)^{2d}
\end{equation}
the claim follows, since the sum over $n$ is $1 + o(1)$ as $c \to \infty$ by comparison with a geometric series.
\end{proof}

In the case that the Markov chains are symmetric random walks on a group, \cite{DS93a} gives the following simplified estimate.  Let $E$ be a symmetric set of generators of a finite group $G$.  For $y \in G$, let $y = z_1z_2 \cdots z_k$ with $z_i \in E$. Denote the least such $k$, $|y|$.  Let $N(z,y)$ denote the number of times which $z$ appears in the chosen representation of $y$.

\begin{theorem}[\cite{DS93a}, Theorem 1]\label{group_comparison_theorem}
 Let $\tilde{p}$ and $p$ be symmetric probabilities on a finite group $G$. Let $E$ be a symmetric set of generators. Suppose that the support of $p$ contains $E$. Then the Dirichlet forms satisfy
 \begin{equation}
  \tilde{\sE} \leq A \sE
 \end{equation}
with
\begin{equation}
 A = \max_{z \in E} \frac{1}{p(z)} \sum_{y \in G} |y|N(z,y) \tilde{p}(y).
\end{equation}

\end{theorem}
This theorem is used together with the following estimate for 3-cycles from \cite{HSZ15} to prove the upper bound of Theorem \ref{upper_bound_theorem}.

\begin{lemma}
 The spectral gap in the regular representation of $\Alt(n)$ for the measure supported uniformly on 3-cycles is $\frac{3}{n-1}$, and the $\frac{\epsilon}{|G|}-\ell^\infty$ mixing time is of order $n \log n$.
\end{lemma}
\begin{proof}
 See \cite{HSZ15}, Appendix A.
\end{proof}

\section{Return times and hitting times}
The return time of a Markov chain at a vertex is the first positive time after leaving the vertex at which the Markov chain again is at the vertex.  The hitting time of a Markov chain to a set $B$ is the first non-negative time at which the chain reaches the set.  We require some estimates for return times and hitting times of simple random walk on the torus $(\zed/n\zed)^2$, as well as the hitting probabilities regarding the likelihood of first reaching individual vertices in sets.

A function $h : \sX \to \bR$ is said to be \emph{harmonic} at $x \in \sX$ if
\begin{equation}
 h(x) = \sum_{y \in \sX} P(x,y) h(y).
\end{equation}
If $(X_t)$ is the sequence of steps of a Markov chain with transition kernel $P$, and $B \subset \sX$, define the \emph{hitting time} $\tau_B$ by
\begin{equation}
 \tau_B = \min\{t \geq 0 : X_t \in B\}.
\end{equation}
In \cite{LPW09} the following proposition is proved, which can be used to estimate return probabilities.
\begin{proposition}[\cite{LPW09}, Proposition 9.1]\label{harmonic_measure_prop}
 Let $(X_t)$ be a Markov chain on a finite state space $\sX$ with irreducible transition matrix $P$, let $B \subset \sX$, and let $h_B: B \to \bR$ be a function defined on $B$.  The function $h: \sX \to \bR$ defined by
 \begin{equation}
  h(x) := \E_x h_B(X_{\tau_B})
 \end{equation}
is the unique extension $h: \sX \to \bR$ of $h_B$ such that $h(x) = h_B(x)$ for all $x \in B$ and $h$ is harmonic for $P$ at all $x \in \sX \setminus B$.
\end{proposition}
The function $h$ is called the `harmonic extension' of $h_B$ to $\sX \setminus B$.

The following lemma can be deduced from Proposition \ref{harmonic_measure_prop}. Let $X_t$ be $\frac{1}{5}$-lazy simple random walk on $(\zed/n\zed)^2$ started from $(1,0)$, and let $p_{y,n}, y \in \{(1,0), (-1,0), (0, 1), (0,-1)\}$ be the probability that $X_t$ first reaches $(0,0)$ from $y$.
\begin{lemma}
  There are limiting probabilities $p_{(1,0)}, p_{(-1,0)}, p_{(0,  1)}=p_{(0,-1)} > 0$ such that $p_{y,n} \to p_y$ as $n \to\infty$.  
\end{lemma}

\begin{proof}
Since simple random walk on $\zed^2$ is recurrent with probability 1, the probability of the first return to 0 taking more than $n$ steps tends to 0 as $n \to \infty$.  Those return paths which take fewer than $n$ steps are the same on $(\zed/n\zed)^2$ as on $\zed^2$, which proves the limit. 
\end{proof}

The return probabilities $p_{x}$ are determined as follows.

\begin{lemma}
Started at $(1,0)$, the return probability to the origin is given by
\begin{align}
p_{(1,0)} &= \frac{1}{2}, \qquad 
 p_{(0,\pm 1)} = \frac{1}{2}-\frac{1}{\pi}
,\qquad  p_{(-1,0)} = \frac{2}{\pi}-\frac{1}{2}.
\end{align}
\end{lemma}

\begin{proof}
 We work first on $(\zed/n\zed)^2$ and then take the limit as $n \to \infty$. Denote $G_{x,n}$ the Green's function on $(\zed/n\zed)^2$ started at $0$ and evaluated at $x$ and $G_{x}$ the Green's function started at 0 and evaluated at $x$ on $\zed^2$. Thus our normalization has $G_{0,n} = G_{0} = 0$.  We define function $h_B$ on set $B$, where $B = \{(0,0),(-1,0)\},$ and extend $h_B$ on $(\zed/n\zed)^2$ by Lemma \ref{harmonic_measure_prop}. 
  Let $h((0,0)) = -G_{(1,0), n}$,  $h((-1,0)) = G_{(1,0),n}$, then
  \begin{align} \label{harmo_green}
  h(x) = G_{x,n} - G_{x-(-1,0),n}
  \end{align}
 By conditioning on the state which is first hit in  $B$, we get 
 \begin{align}
  h(x) = p_{B_{(0,0)},n}(x)h((0,0)) + p_{B_{(-1,0)},n}(x)h((-1,0)),
 \end{align}
  where $p_{B_{(0,0)},n}(x)$ is, starting at point $x$ the probability of hitting the origin when first hitting set $B$ and $p_{B_{(-1,0)},n}(x)$ is the probability of hitting $(-1,0)$ when first hitting set $B$. For convenience, let $p_{B_{(0,0)},n}$ and $p_{B_{(-1,0)},n}$ denote $p_{B_{(0,0)},n}(1,0)$ and $p_{B_{(-1,0)},n}(1,0)$. 
 
Since we know $p_{B_{(0,0)},n} + p_{B_{(-1,0)},n} = 1$,  plug in $(0,1)$ in (\ref{harmo_green}) and we get 
 \begin{align}
      p_{B_{(0,0)},n} &= \frac{G_{(2,0),n}}{2G_{(1,0),n}},
\qquad    p_{B_{(-1,0)},n} = 1 - \frac{G_{(2,0),n}}{2G_{(1,0),n}}.
 \end{align}
 Plugging in $(0,1)$ in (\ref{harmo_green}), we get 
 \begin{align}
      p_{B_{(-1,0)},n}((0,1)) = 1 - \frac{G_{(1,1),n}}{2G_{(1,0),n}}.
 \end{align}
 
 Due to symmetry, we have
 \begin{equation} p_{(0,1),n} = p_{B_{(-1,0)},n}((1,0))p_{(1,0),n} \end{equation}
 i.e. start at $(1,0)$, calculate $p_{(0,1),n}$ by conditioning on the state of hitting $(0,1)$ before hitting the origin.
 
 Let $p'_{(1,0),n}$, $p'_{(0,1),n}$, $p'_{(0,-1),n}$ denote the probability of returning to origin through $(1,0),(0,1),(0,-1),$ without passing through point $(-1,0)$. We have
 \begin{align}
p_{B_{(0,0)},n} &= p'_{(0,1),n} + p'_{(0,-1),n} + p'_{(1,0),n}= 2p'_{(0,1),n} + p'_{(1,0),n}  
\\\notag p_{(-1,0),n} &= p_{B_{(-1,0)},n}p_{(1,0),n}\\ \notag
p_{(1,0),n} &= p'_{(1,0),n} + p_{B_{(-1,0)},n}p_{(-1,0),n} = \frac{p'_{(1,0),n}}{1-p^2_{B_{(-1,0)},n}}\\ \notag
p_{(0,1),n} &= p'_{(0,1),n} + p_{B_{(-1,0)},n}p_{(0,1),n} = \frac{p'_{(0,1),n}}{1-p_{B_{(-1,0)},n}}.
 \end{align} 
By solving the above linear system, we get the desired quantity of return probabilities,
\begin{align}
p_{(1,0),n} &= \frac{2G_{(1,0),n}}{8G_{(1,0),n}-2G_{(1,1),n}-G_{(2,0),n}}\\
\notag p_{(0,\pm 1),n} &= \frac{2G_{(1,0),n}-G_{(1,1),n}}{8G_{(1,0),n}-2G_{(1,1),n}-G_{(2,0),n}}
\\ \notag p_{(-1,0),n} &= \frac{2G_{(1,0),n}-G_{(2,0),n}}{8G_{(1,0),n}-2G_{(1,1),n}-G_{(2,0),n}}.
\end{align}
Letting $n \to \infty$,  $G_{x,n} \to G_{x}$.  The exact values were calculated in Mathematica. 
 
\end{proof}

The following proposition is a variant of one proved in \cite{LPW09}, adapted to the case of $\frac{1}{5}$-lazy simple random walk.
\begin{proposition}[\cite{LPW09}, Proposition 10.13]\label{stopping_time_prop}
 Let $x$ and $y$ be two points at distance $k \geq 1$ on the torus $(\zed/n\zed)^2$, and let $T_y$ be the time of the first visit to $y$ after leaving $x$.  There exist constants $0 < c_2 \leq C_2 <\infty$ such that 
 \begin{equation}
  c_2 n^2 \log(k) \leq \E_x(T_y) \leq C_2 n^2 \log(k+1).
 \end{equation}
Also, the exact formula $\E_x(T_x) = \frac{5}{4}n^2$ holds.
 \end{proposition}
 \begin{proof}
  When considering non-lazy simple random walk, which transitions to each neighbor with probability $\frac{1}{4}$, the exact expected return time is $n^2$ and the lower and upper bounds on the hitting time are proved in \cite{LPW09}. Under $\frac{1}{5}$-lazy simple random walk, the expected time to make each move is $\frac{5}{4}$, which proves the claims.
 \end{proof}

 \begin{lemma} \label{general_random_walk_hitting_time}
  Let $\mu$ be a symmetric probability measure on $(\zed/n\zed)^2$ satisfying, for some $c>0$, $\mu((0,0)), \mu((\pm 1, 0)), \mu((0, \pm 1)) \geq c$.  The expected hitting time of simple random walk driven by $\mu$ to any $x \in (\zed/n\zed)^2$ is $O(n^2\log n)$.
 \end{lemma}
\begin{proof} It follows from \cite{LPW09}, chapters 9 and 10 that the expected hitting time from $x$ to $y$ on $(\zed/n\zed)^2$ is $O(n^2 \sR(x,y))$ where $\sR(x,y)$ is the effective resistance. Moreover, by Thomson's Principle, $\sR(x,y) = \inf\{\sE[\theta], \theta \text{ a unit flow $x,y$}\}$ and $\sE[\theta] = \sum_e [\theta(e)]^2r(e)$, where $r(e)$ is the resistance of an edge $e$.  Since the resistance of the nearest neighbor edges under the random walk $\mu$ is bounded to within constants by that for simple random walk, $\sR(x,y) = O(\log n)$.  
\end{proof}

 By Markov's inequality, for any $x \in (\zed/n\zed)^2$, there is a constant $C>0$ such that $\Prob_{x}(\tau_0 > Cn^2 \log n) < \frac{1}{2}$.  Applying this iteratively proves that $\Prob_x (\tau_0 > cn^2 \log n)$ decays exponentially in $c$ as $c \to \infty$.  In particular,  $\Var_x(\tau_y) \ll n^4 (\log n)^2$.

 Another useful expression for the hitting time can be obtained as a generating function.  Let $P$ be the transition kernel of $\frac{1}{5}$-lazy simple random walk on $(\zed/n\zed)^2$, and let $P'$ be $P$ with row and column corresponding to $(0,0)$ deleted. Let 
  \begin{align}
  R(z) &= (I-zP')^{-1}= I + zP' + (zP')^2 + ...
 \end{align}
be the resolvent. 
 
 \begin{lemma}\label{char_fun_lemma}
  The characteristic function of the hitting time from $(1,0)$ to $(0,0)$ under $\frac{1}{5}$-lazy simple random walk at $z = e^{2\pi i \xi}$ is
  \begin{align}
\chi(z) =  \frac{z}{5} e_{(1,0)}^t R(z) \left(e_{(1,0)} + e_{(-1,0)} + e_{(0,1)} + e_{(0,-1)} \right).
\end{align}
The expected hitting time is
\begin{equation}
 1 + \frac{1}{5} e_{(1,0)}^t R'(1) \left(e_{(1,0)} + e_{(-1,0)} + e_{(0,1)} + e_{(0,-1)} \right)
\end{equation}
which is of order $n^2$.
 \end{lemma}
\begin{proof}
 $(P')^n$ enumerates the transition probabilities that result from length $n$ paths which do not visit $(0,0)$. To obtain the characteristic function formula, condition on the number of steps taken under $P'$, and then use that the probability of transitioning from one of the neighbors of $(0,0)$ to $(0,0)$ is $\frac{1}{5}$. 
 
 We have $\chi(1)=1$, since the walk hits $(0,0)$ in finite time with probability 1.
 
 The formula for the expected hitting time holds since the expectation is $\chi'(1)$.
\end{proof}

Since $P'$ is symmetric, it can be diagonalized using an orthonormal set of eigenvectors.  Let the corresponding eigenvalues be $\lambda_1 \geq \lambda_2 \geq \cdots \geq \lambda_{n^2-1}$ with eigenvectors $v_1, v_2, ..., v_{n^2-1}$.  We have $v_1$ is non-negative and $1 > \lambda_1$.  Write
\begin{equation}
 R(z) = \sum_{i=1}^{n^2-1} \frac{v_i v_i^t}{1-z\lambda_i}.
\end{equation}
Let $c_{i,x} = \langle v_i, e_x\rangle$.

\begin{lemma}
 The largest eigenvalue $\lambda_1$ satisfies $\frac{1}{n^2\log n} \ll 1-\lambda_1 \ll \frac{1}{n^2}$.  Also,
 \begin{align}\label{1st_moment}
  \frac{5}{4} \leq \sum_{i=1}^{n^2-1} \frac{c_{i, (1,0)}^2}{1-\lambda_i} < 5
 \end{align}
 and 
 \begin{align}\label{2nd_moment}
  \sum_{i=1}^{n^2-1} \frac{c_{i, (1,0)}^2}{(1-\lambda_i)^2}\asymp n^2.
 \end{align}
Furthermore, there is a constant $c>0$ such that
\begin{align}\label{large_eigenvalue_part}
 \sum_{i: (1-\lambda_i) > \frac{c}{n}} \frac{c_{i, (1,0)}^2}{(1-\lambda_i)} \gg \frac{1}{\log n}.
\end{align}

\end{lemma}
\begin{proof}
 Let $v_1$ be in the top eigen-space, $1^t v_1 = 1$ so that $v_1$ is a probability vector. Since $(P')^m v_1 = \lambda_1^m v_1$, $1^t (P')^m v_1 = \lambda_1^m$ is the probability of not reaching $(0,0)$ in $m$ steps, started from a distribution proportional to $v_1$.  Since the expected hitting time to $(0,0)$ is $O(n^2 \log n)$ uniformly in the starting point, it follows that for some $c>0$,  $\lambda_1 \leq 1 - \frac{c}{n^2\log n}$.
 
 The bound $\lambda_1 > 1 - \frac{c}{n^2}$ will follow after establishing (\ref{2nd_moment}), since $\sum_i c_{i,(1,0)}^2 = 1$ by orthogonality.  To prove (\ref{1st_moment}) note that $\chi(1) = 1$ may be written 
 \begin{equation}
 1 = \frac{1}{5} \sum_{i=1}^{n^2-1} \frac{c_{i,(1,0)} (c_{i, (1,0)} + c_{i, (-1,0)} + c_{i, (0,1)} + c_{i, (0,-1)}) }{1-\lambda_i}.
\end{equation}
Furthermore, by symmetry, $\sum_i \frac{c_{i,x}^2}{1-\lambda_i}$ is independent of $x \in \{(\pm 1, 0), (0, \pm 1)\}$, so that 
\begin{align}
\sum_{i=1}^{n^2-1} \frac{c_{i, (1,0)}^2}{1-\lambda_i} &\leq  \sum_{i=1}^{n^2-1} \frac{c_{i,(1,0)} (c_{i, (1,0)} + c_{i, (-1,0)} + c_{i, (0,1)} + c_{i, (0,-1)}) }{1-\lambda_i}\\\notag& \leq 4\sum_{i=1}^{n^2-1} \frac{c_{i, (1,0)}^2}{1-\lambda_i}.
\end{align}
Similarly, the expected hitting time formula may be written
\begin{equation}
 \sum_{i=1}^{n^2-1}\frac{c_{i,(1,0)} (c_{i, (1,0)} + c_{i, (-1,0)} + c_{i, (0,1)} + c_{i, (0,-1)})\lambda_i }{(1-\lambda_i)^2} \asymp n^2.
\end{equation}
Those terms with $\frac{1}{1-\lambda_i}$ bounded contribute $O(1)$ to this sum, so that the negative terms may be dropped, and hence 
\begin{equation}
 \sum_{i=1}^{n^2-1}\frac{c_{i,(1,0)} (c_{i, (1,0)} + c_{i, (-1,0)} + c_{i, (0,1)} + c_{i, (0,-1)}) }{(1-\lambda_i)^2} \asymp n^2.
\end{equation}
Again by symmetry, and Cauchy-Schwarz, 
\begin{equation}
 \sum_{i=1}^{n^2-1} \frac{c_{i, (1,0)}^2}{(1-\lambda_i)^2} \asymp n^2.
\end{equation}
To prove (\ref{large_eigenvalue_part}), note that if $c>0$ is sufficiently small, since $\sum_i c_{i, (1,0)}^2 = 1$,
\begin{equation}
 \sum_{i: (1-\lambda_i) > \frac{c}{n}} \frac{c_{i, (1,0)}^2}{(1-\lambda_i)^2} \gg n^2.
\end{equation}
The claim now follows from $\frac{1}{1-\lambda_i} \ll n^2 \log n$.
\end{proof}

The following bounds are useful in bounding the characteristic function of the hitting time.
\begin{lemma}\label{hitting_time_char_fun_bound}
 Let $\vartheta$ be maximal such that 
 \begin{equation}
  \sum_{i: (1-\lambda_i) > \vartheta} \frac{c_{i, (1,0)}^2}{1-\lambda_i} \geq \left(1 - \frac{1}{(\log n)^3}\right) \sum_{i=1}^{n^2-1} \frac{c_{i, (1,0)}^2}{1-\lambda_i}.
 \end{equation}
 Then for $\vartheta < \xi \leq \frac{1}{2}$,
 \begin{align}
  \sum_{i: 1-\lambda_i < \xi} c_{i, (1,0)}^2 \left( \frac{1}{1-\lambda_i} - \frac{1}{|1-\lambda_i e^{2\pi i \xi}|}\right) \gg \frac{1}{(\log n)^3}.
 \end{align}
 For $0 < \xi \leq \vartheta$,
 \begin{equation}
  \sum_{i: (1-\lambda_i) > \vartheta} c_{i, (1,0)}^2 \left(\frac{1}{1-\lambda_i} - \frac{1}{|1-\lambda_i e^{2\pi i \xi}|} \right) \gg \xi^2 n^4.
 \end{equation}

\end{lemma}
\begin{proof}
By the previous lemma, $\vartheta \ll \frac{1}{n}$. We have
\begin{equation}
 \frac{1}{|1 -\lambda_i e^{2\pi i \xi}|} \leq \frac{1}{| (1-\lambda_i) + i \lambda_i \sin(2\pi \xi)| } = \frac{1}{1-\lambda_i}  \frac{1}{\left| 1 + \frac{i \lambda_i \sin(2\pi \xi)}{1-\lambda_i}\right|}.
\end{equation}
and, thus,
\begin{align}
 \frac{1}{1-\lambda_i} - \frac{1}{|1-\lambda_i e^{2\pi i \xi}|} &\gg \frac{1}{1-\lambda_i} \left(1 - \frac{1}{\sqrt{1 + \frac{\xi^2}{(1-\lambda_i)^2}}}\right)\\
 &\notag \gg \frac{1}{1-\lambda_i} \min\left(1, \frac{\xi^2}{(1-\lambda_i)^2}\right).
\end{align}

We now conclude, when $\vartheta < \xi \leq \frac{1}{2}$,
\begin{align}
 \sum_{i: (1-\lambda_i) < \xi}c_{i, (1,0)}^2 \left( \frac{1}{1-\lambda_i} - \frac{1}{|1-\lambda_i e^{2\pi i \xi}|}\right) \gg \sum_{i: (1-\lambda_i) < \xi}\frac{c_{i, (1,0)}^2}{1-\lambda_i} \gg \frac{1}{(\log n)^3}.
\end{align}
For $0 < \xi \leq \vartheta$, note that
\begin{equation}
 \sum_{i: (1-\lambda_i) < \vartheta} \frac{c_{i, (1,0)}^2}{(1-\lambda_i)^3} \ll \frac{1}{(1-\lambda_1)^2} \sum_{i: (1-\lambda_i) < \vartheta} \frac{c_{i, (1,0)}^2}{1-\lambda_i} \ll \frac{n^4}{\log n}
\end{equation}
while, by H\"{o}lder,
\begin{equation}
 \sum_{i=1}^{n^2-1} \frac{c_{i, (1,0)}^2}{(1-\lambda_i)^3} \gg n^4.
\end{equation}
Hence,
\begin{equation}
 \sum_{i: 1-\lambda_i > \vartheta}c_{i, (1,0)}^2 \left(\frac{1}{1-\lambda_i} - \frac{1}{|1-\lambda_i e^{2\pi i \xi}|} \right) \gg \xi^2 \sum_{i: 1-\lambda_i > \vartheta}\frac{c_{i, (1,0)}^2 }{(1-\lambda_i)^3} \gg \xi^2 n^4.
\end{equation}

\end{proof}

\section{The mixing of a single piece, proof of Theorem \ref{single_piece_mixing_theorem}}
In this section we track the location of a single numbered piece $\sP$, along with the empty piece $\sP_e$.

Let $t_0$ be the first time at which $\sP_e$ swaps positions with $\sP$ from the above or below, and define a sequence of stopping times $t_1, t_2, t_3, ...$ in which $t_{2i-1}$ is the first time after $t_{2i-2}$ at which $\sP_e$ swaps positions with $\sP$ from left or right, and $t_{2i}$ is the first time after $t_{2i-1}$ at which $\sP_e$ swaps positions with $\sP$ from the above or below.  

Let $H_0$ be the number of positions (left is negative, right is positive) that $\sP$ moves prior to $t_1$ and $V_0$ the number of positions (up is positive, down is negative) that $\sP$ moves prior to $t_1$.

For $i \geq 1$,  let $H_i$ be the number of horizontal moves of $\sP$ in $[t_{2i-1},t_{2i})$ and let $V_i$ be the number of vertical moves of $\sP$ in $[t_{2i},t_{2i+1})$.  
 
\begin{lemma}
 The collection of random variables $\{H_i, V_i\}_{i=1}^\infty$ are i.i.d. symmetric, mean 0, and have exponentially decaying tails.  They are independent of $H_0, V_0$, and these variables have exponentially decaying tails.
\end{lemma}
\begin{proof}
 The i.i.d. property follows from the strong Markov property,  see e.g. \cite{D19}, Chapter 4. Symmetry follows from the symmetry of the walk, while the exponentially decaying tails property follows since $p_{(0, \pm 1), n} \to p_{(0, \pm 1)} > 0$, so that $H_i$ and $V_i$ are the sum of a geometric variable number of moves with parameter bounded away from 1.
\end{proof} 
 
 Let, for $i \geq 1$,
\begin{equation}
 r_i = t_{2i}-t_{2i-1}, \qquad s_i = t_{2i+1}-t_{2i}.
\end{equation}
The collection $\{(H_i, r_i), (V_i, s_i)\}_{i=1}^\infty$ are also i.i.d. and are independent of $\{H_0, V_0, t_0, t_1\}$. Set 
\begin{equation}
 s_n^2 = \E[H_1^2], \; \mu_n = \E[r_1],\;  v_n^2 = \Var[r_1].
\end{equation}

Let $m_i$ be the number of times $\sP$ moves either left or right between $t_{2i-1}$ and $t_{2i}$ and $n_i$ the number of times $\sP$ moves either up or down between $t_{2i}$ and $t_{2i+1}$.  Call a type I return of $\sP_e$ a sequence of moves in which $\sP_e$ begins adjacent to $\sP$, ends at the next time $\sP_e$ swaps position with $\sP$ from the same direction (horizontal or vertical), and does not swap positions with $\sP_e$ from the opposite direction in between.  Call a type II return a sequence of moves of $\sP_e$ in which it begins adjacent to $\sP$, ends adjacent to $\sP_e$ from the opposite direction, and does not swap position with $\sP_e$ in between, but will swap positions with $\sP_e$ from the opposite direction on the next move. Thus $r_i$ is the sum of the length of $m_i$ independent type I returns and one type II return, and $s_i$ is the sum of the lengths of $n_i$ independent type I returns and one type II return.  

The number of type I returns has a geometric distribution of bounded parameter.

\begin{lemma}\label{asymptotics_lemma}
We have
\begin{equation}
 \lim_{n \to \infty} s_n^2 = s^2, \qquad \lim_{n \to \infty} \frac{\mu_n}{n^2} = \mu 
\end{equation}
with 
\begin{equation}
 s^2 =\frac{1}{2p_{(0,\pm 1)}} \frac{1 - p_{(1,0)} + p_{(-1,0)}}{1 + p_{(1,0)} - p_{(-1,0)}}, \qquad \mu  = \frac{5}{4} \left( \frac{1}{2p_{(0,\pm 1)}} \right).
\end{equation}
Also, $v_n = O(n^2 \log n)$. 
\end{lemma}
\begin{rem}
 The constant $c_{\puzzle}$ is
 \begin{equation}
  c_{\puzzle} = \frac{2\mu}{s^2} = \frac{5}{2} \frac{1 + p_{(1,0)}-p_{(-1,0)}}{1-p_{(1,0)} +p_{(-1,0)}} = \frac{5}{2}(\pi-1).
 \end{equation}
\end{rem}

 \begin{proof} 
 The $H_1$ process makes one move right or left, with equal probability, then makes $k \geq 0$ right or left returns followed by an up or down return.  Since we consider $H_1^2$, by symmetry assume that the initial move is $-1$, so that afterwards, $\sP_e$ is at position $\sP + (1,0)$. 
 
 Condition on $k$ and let $Z_k$ be the conditional displacement of the moves following the first one.
 Thus $Z_0 = 0$ and, for $k \geq 1$, 
 \begin{equation}Z_k = \left\{\begin{array}{ccc} 1 - Z_{k-1} && \text{prob. } \frac{p_{(1,0),n}}{1-2p_{(0, \pm 1),n}}\\ -1 + Z_{k-1} && \text{prob. } \frac{p_{(-1,0),n}}{1-2p_{(0, \pm 1),n}} \end{array}\right..\end{equation}
 Hence 
 \begin{equation}
  \E[Z_k] = \frac{p_{(1,0),n} - p_{(-1,0),n}}{1-2p_{(0,\pm 1),n}}(1 - \E[Z_{k-1}]).
 \end{equation}
Solving the recurrence, it follows that 
\begin{equation}
 \E[Z_k] = \sum_{i=1}^k (-1)^{i-1} \left( \frac{p_{(1,0),n} - p_{(-1,0),n}}{1-2p_{(0,\pm 1),n}}\right)^i.
\end{equation}
Similarly,
\begin{equation}
 \E[Z_k^2] = 1 - \E[2Z_{k-1}] + \E[Z_{k-1}^2]
\end{equation}
from which it follows that 
\begin{align}
 \E[Z_k^2] &= k - 2 \sum_{j = 1}^{k-1} \sum_{i=1}^j (-1)^{i-1}\left( \frac{p_{(1,0),n} - p_{(-1,0),n}}{1-2p_{(0,\pm 1),n}}\right)^i\\
 \notag &= k-2\sum_{i=1}^{k-1} (-1)^{i-1} (k-i)\left( \frac{p_{(1,0),n} - p_{(-1,0),n}}{1-2p_{(0,\pm 1),n}}\right)^i .
\end{align}
Conditioning on the number $k$ of type I returns,
\begin{align}
 \E[H_1^2] &= \sum_{k = 0}^\infty \E[(-1 + Z_k)^2] 2p_{(0, \pm 1),n} (1 -2p_{(0, \pm 1),n})^k\\
 \notag &= \sum_{k=0}^\infty 2p_{(0, \pm 1),n}(1-2p_{(0, \pm 1),n})^k - 2 \sum_{k=1}^\infty \E[Z_k]2p_{(0, \pm 1),n}(1-2p_{(0, \pm 1),n})^k \\ \notag &+ \sum_{k=1}^\infty \E[Z_k^2]2p_{(0, \pm 1),n}(1-2p_{(0, \pm 1),n})^k.
 \end{align}
 We have
 \begin{equation}
  \sum_{k=0}^\infty 2p_{(0, \pm 1),n}(1-2p_{(0, \pm 1),n})^k = 1
 \end{equation}
and
\begin{align}
 &\sum_{k=1}^\infty \E[Z_k]2p_{(0, \pm 1),n}(1-2p_{(0, \pm 1),n})^k\\ \notag
 &=2p_{(0, \pm 1),n} \sum_{k=1}^\infty \sum_{i=1}^k (-1)^{i-1} (p_{(1,0),n} - p_{(-1,0),n})^i (1 - 2p_{(0, \pm 1),n})^{k-i}\\ \notag
 &= \sum_{i=1}^\infty (-1)^{i-1} (p_{(1,0),n} - p_{(-1,0),n})^i\\ \notag
 &= \frac{ p_{(1,0),n} - p_{(-1,0),n}}{1 + p_{(1,0),n} - p_{(-1,0),n}}.
\end{align}
Also,
 \begin{align}
 &\sum_{k=1}^\infty \E[Z_k^2]2p_{(0, \pm 1),n}(1-2p_{(0, \pm 1),n})^k\\ \notag &=2 p_{(0, \pm 1),n} \sum_{k=1}^\infty k (1-2p_{(0, \pm 1),n})^k \\ \notag &- 4p_{(0, \pm 1)} \sum_{k=1}^\infty\left(1-2p_{(0, \pm 1),n} \right)^k \sum_{i=1}^{k-1} (-1)^{i-1} (k-i)\left( \frac{p_{(1,0),n} - p_{(-1,0),n}}{1-2p_{(0,\pm 1),n}}\right)^i \\ \notag
 &= 2 p_{(0, \pm 1),n} \sum_{k=1}^\infty k (1-2p_{(0, \pm 1),n})^k \\ \notag &- 4p_{(0, \pm 1),n} \sum_{i=1}^\infty (-1)^{i-1} (p_{(1,0),n} - p_{(-1,0),n})^i\sum_{k=1}^\infty k (1-2p_{(0, \pm 1),n})^k\\ \notag
 &= \frac{1 - 2p_{(0, \pm 1),n}}{2 p_{(0, \pm 1),n}} - 2\frac{1 - 2p_{(0, \pm 1),n}}{2 p_{(0, \pm 1),n}} \frac{ p_{(1,0),n} - p_{(-1,0),n}}{1+ p_{(1,0),n} - p_{(-1,0),n}} .
\end{align}
Combining the above obtains
\begin{equation}
 \E[H_1^2] = \frac{1}{2p_{(0,\pm 1),n}} \frac{1 - p_{(1,0),n} + p_{(-1,0),n}}{1 + p_{(1,0),n} - p_{(-1,0),n}}.
\end{equation}

 The number of times that the piece moves between $t_{2i-1}$ and $t_{2i}$ is $1 + m_i$.  We have, by the law of large numbers, 
 \begin{equation}
 \frac{1}{N} \sum_{i=1}^N m_i \to \E[m_1] = \sum_{k=0}^\infty k (2p_{(0, \pm 1),n})(1-2p_{(0,\pm 1),n})^k = \frac{1}{2 p_{(0, \pm 1),n}} - 1.
 \end{equation}
 Also, $\frac{1}{N} \sum_{i=1}^{N} r_i \to \E[r_1]$. Similarly, the averages converge for $n_i$ and $s_i$, which have the same distribution. Since, on average, the piece moves once every $\frac{5}{4} (n^2-1)$ steps of the walk, $\frac{\E[r_1]}{1 + \E[m_1]} = \frac{5}{4}(n^2-1)$, or
  \begin{equation}\E[r_1] = \frac{5}{4}(n^2-1) \left( \frac{1}{2p_{(0,\pm 1),n}} \right).\end{equation}
 \end{proof}

The primary step in establishing the mixing of the piece $\sP$ is establishing the asymptotic independence of the coordinates of the sum 
\begin{equation}
 S_N = \left( \sum_{i=1}^N H_i, \sum_{i=1}^N V_i, \sum_{i=1}^N (r_i + s_i)\right)
\end{equation}
as $N \to \infty$. This is demonstrated by considering the characteristic function 
\begin{equation}
 \chi(\xi_1, \xi_2) = \E\left[e^{2\pi i \frac{\xi_1}{n} H_1 + 2\pi i \xi_2 r_1} \right], \qquad \xi_1 \in \zed/n\zed, \;\xi_2 \in \bR/\zed.
\end{equation}
 
 \begin{lemma}
 There is a constant $c>0$ such that, uniformly in $n$ and uniformly in $\xi_1 \in \left(-\frac{n}{2}, \frac{n}{2}\right]$, and $\xi_2 \in \left(-\frac{1}{2}, \frac{1}{2}\right]$,
 \begin{equation}
  \left|\chi(\xi_1, \xi_2)\right| \leq 1 - c \max\left(\frac{\xi_1^2}{n^2}, \xi_2^2\right).
 \end{equation}

\end{lemma}

\begin{proof}
 Consider paths of either fixed, bounded length, or fixed displacement of $\sP$ but bounded varying length.  All of these have positive probability, which is not dependent on $n$ for all $n$ sufficiently large. The variation in their phase is of the given magnitude.
\end{proof}

We now give an exact matrix formula for $\chi(\xi_1, \xi_2)$. Recall that we let $P$ be the transition matrix of $\frac{1}{5}$-lazy simple random walk on $(\zed/n\zed)^2$, and that $P'$ is the minor excluding the row and column of $(0,0)$, $R(z) = (I-zP')^{-1}$.  Let $M(z_1, z_2)$ be the transition matrix on $(\zed/n\zed)^2 \setminus \{(0,0)\}$  with $\frac{1}{5} z_1z_2$ in the transition from $(1,0)$ to $(-1,0)$ and $\frac{1}{5} \frac{z_2}{z_1}$ in the transition from $(-1,0)$ to $(1,0)$, and zeros elsewhere.  Let $w(z_1,z_2)$ be the vector with entries $\frac{z_1z_2}{2}$ at $(-1,0)$ and $\frac{z_2}{2z_1}$ at $(1,0)$, with zeros elsewhere, and let $v$ be the vector with value $\frac{1}{5}$ at $(0, \pm 1)$ and zeros elsewhere.
 \begin{lemma}
  Let $\xi_1 \in \zed/n\zed$, $\xi_2 \in \bR/\zed$, and set $z_1 = e^{\frac{2\pi i \xi_1}{n}}, z_2 = e^{2\pi i \xi_2}$.
  Then
 \begin{equation}
  \chi(\xi_1, \xi_2) = w(z_1, z_2)^t (I-R(z_2) M(z_1, z_2))^{-1} R(z_2) v.
 \end{equation}
 \end{lemma}
\begin{proof}
 The sequence of moves described in phase space by $\chi(\xi_1, \xi_2)$ are as follows.  An initial move, which involves one move of the empty square and one update, right or left of $\sP$ occurs.  This is recorded by $w(z_1, z_2)^t$ in which if the empty square swaps places from the right, it now has the position to the left $(-1,0)$ of $\sP$ and makes one move, hence contributes $z_1z_2$ to the phase. If instead the empty square swaps places from the left then it now occupies $(1,0)$ relative to $\sP$ and contributes $\frac{z_2}{z_1}$ to the phase.
 
 Now there are 0 or more excursions of the empty square followed by a right or left move of $\sP$.  A right or left move of $\sP$ entails moving the empty square from the position on the right of $\sP$ to the position on the left, or vice versa.  This is captured by $M(z_1, z_2)$.   Finally there is a final excursion, captured by $R$, which is finished by moving onto $\sP$ from above or below, captured by $v$.
\end{proof}

The previous lemma implies the following bound. 

\begin{lemma}
Uniformly in $n$, $\xi_1 \in (-\frac{n}{2}, \frac{n}{2}]$ and $\xi_2 \in (-\frac{1}{2}, \frac{1}{2}]$,
 \begin{equation}
  1-|\chi(\xi_1, \xi_2)| \gg e_{(1,0)}^t R(1) e_{(1,0)}- \left|e_{(1,0)}^t R(e^{2\pi i \xi_2}) e_{(1,0)} \right|.
 \end{equation}

\end{lemma}
\begin{proof}
 In the term $(I-R(z_2) M(z_1, z_2))^{-1} = \sum_{k=0}^\infty (R(z_2)M(z_1,z_2))^k$, there is a probability, bounded uniformly away from 0 that the $k = 1$ term is taken, and a probability bounded uniformly from 0 that the return is of type $e_{(1,0)}$ to $e_{(1,0)}$.  With no cancellation, the sum of path probabilities making up the return is   $e_{(1,0)}^t R(1) e_{(1,0)}$, while with the phase, the sum has size $\left|e_{(1,0)}^t R(e^{2\pi i \xi_2}) e_{(1,0)} \right|$.
\end{proof}
\begin{lemma} There is a constant $c>0$ such that, for all $\xi_1 \in \zed/n\zed$ and $-\frac{1}{2} < \xi_2 \leq \frac{1}{2}$,
 \begin{equation}
  |\chi(\xi_1, \xi_2)|\leq  1 - c\min\left(\frac{1}{(\log n)^3},  \xi_2^2 n^4 \right).
 \end{equation}

\end{lemma}
\begin{proof}
 By the previous lemma,
 \begin{align}
  1-|\chi(\xi_1, \xi_2)| &\gg \sum_{i=1}^{n^2-1} \frac{c_{i, (1,0)}^2}{1-\lambda_i} - \left|\sum_{i=1}^{n^2-1} \frac{c_{i, (1,0)}^2}{1-\lambda_i e^{2\pi i\xi_2}} \right| \\ \notag 
  &\geq \sum_{i=1}^{n^2-1} c_{i, (1,0)}^2 \left(\frac{1}{1-\lambda_i} - \frac{1}{|1-\lambda_i e^{2\pi i\xi_2}|} \right).
 \end{align}
By Lemma \ref{hitting_time_char_fun_bound}, it follows that
\begin{equation}
 1-|\chi(\xi_1, \xi_2)| \gg \min\left(\frac{1}{(\log n)^3}, \xi_2^2 n^4\right).
\end{equation}

\end{proof}
At small frequencies, the characteristic function may be estimated by Taylor expansion. Recall $\E[H_1^2] = s_n^2$, $\E[r_1] = \mu_n$ and $\Var[r_1] = v_n^2$.
\begin{lemma}
 For $\xi_1 \in \zed/n\zed$, $|\xi_1| \leq \frac{n}{2}$ and for complex $\xi_2$,  $|\xi_2| \ll \frac{1}{n^2 (\log n)^2}$,
 \begin{align}
  \chi(\xi_1, \xi_2) &= \exp\left(2\pi i \xi_2 \mu_n- \frac{2\pi^2 \xi_1^2}{n^2} s_n^2 - 2\pi^2 \xi_2^2 v_n^2\right)\\ \notag &\times \left(1  + O\left(\frac{\xi_1^3}{n^3} + \xi_2^3 (n^2 \log n)^3 \right)\right).
 \end{align}

\end{lemma}
\begin{proof}
To obtain this estimate, write
\begin{equation}
 \chi(\xi_1, \xi_2) = e^{2\pi i \xi_2 \mu_n} \E\left[e^{2\pi i\left( \frac{\xi_1}{n} H_1 + \xi_2(r_1 - \mu_n)\right)}\right]. 
\end{equation}
Now Taylor expand $e^{2\pi i x} = 1 + 2\pi i x - 2\pi^2 x^2 + O(x^3)$ and use the moments
\begin{align}
 \E[H_1] = 0,\qquad  \E[H_1^2] = s_n, \qquad \E[|H_1|^3] = O(1)
\end{align}
and 
\begin{align}
 \E[r_1 - \mu_n] &= 0,\\ \notag \E[(r_1 - \mu_n)^2] &= v_n, \\ \notag \E[H_1 (r_1 - \mu_n)] &= 0,\\ \notag \E[|r_1 - \mu_n|^3] &= O((n^2 \log n)^3).
\end{align}

\end{proof}
The following local limit theorem is the main technical ingredient in this section.
\begin{theorem}\label{llt}
 Let $n \geq 2$, $(\log n)^{19} \leq N \leq n^3$, $|t-2N\mu_n| < \sqrt{N}v_n \log n$, for any $A>0$,
 \begin{align}
 & \Prob(S_N = (i,j, t))= O\left(\frac{( \log n)^6}{N^{2}v_n} \exp\left(-\frac{(t-2N\mu_n)^2}{4Nv_n^2} \right)\right)+O_A(n^{-A})
 \\\notag&+ \frac{\exp\left(-\frac{(t-2N\mu_n)^2}{4Nv_n^2} \right)}{\sqrt{4\pi N}v_n}\left(\frac{1}{n^2} \sum_{\substack{\xi \in (\zed/n\zed)^2\\ \|\xi\|_2 \ll \frac{n}{\sqrt{N}} \log n}}e^{-2\pi i\frac{\xi \cdot(i,j)}{n}} \exp\left(-\frac{2\pi^2 \|\xi\|_2^2 s_n^2 N}{n^2} \right) \right).
 \end{align}

\end{theorem}
\begin{proof}
 The characteristic function of $S_N$ at frequencies $\frac{\xi_1}{n}, \frac{\xi_2}{n}, \xi_3$ is given by $\chi(\xi_1, \xi_3)^N \chi(\xi_2, \xi_3)^N$.  Hence, by Fourier inversion,
 \begin{align}
  &\Prob(S_N = (i,j, t)) 
  \\ \notag&= \frac{1}{n^2}\sum_{\xi = (\xi_1, \xi_2) \in (\zed/n\zed)^2} \int_{\xi_3 \in \bR/\zed} e^{-2\pi i \left(\frac{\xi_1 i + \xi_2 j}{n} +\xi_3 t\right)} \chi(\xi_1, \xi_3)^N \chi(\xi_2, \xi_3)^N d\xi_3.
 \end{align}
Using the bound $|\chi(\xi_1, \xi_3)| \leq 1 - c \max\left(\frac{\xi_1^2}{n^2}, \xi_3^2 \right)$, truncate the torus variables to $\|\xi\|_2 \ll \frac{n}{\sqrt{N}} \log n$ with error, for any $A> 0$, $O_A(n^{-A})$.  Next using the bound $|\chi(\xi_1, \xi_3)| \leq 1 - c \min\left(\frac{1}{(\log n)^3}, \xi_3^2 n^4\right)$, truncate the $\xi_3$ integral to  $|\xi_3| \ll \frac{(\log n)^2}{n^2 \sqrt{N}}$ with the same error. Inserting the Taylor expansion for $\chi(\xi_1, \xi_3)$ and $\chi(\xi_2, \xi_3)$ at low frequencies,
\begin{align}
 &\Prob(S_N = (i,j,t)) = \frac{1}{n^2} \sum_{\|\xi\|_2 \ll \frac{n}{\sqrt{N}} \log n}e^{-2\pi i\left(\frac{\xi_1 i + \xi_2 j}{n} \right)}\\
 \notag
 &  \times \int_{|\xi_3| \ll \frac{(\log n)^2}{n^2 \sqrt{N}}} e^{-2\pi i\xi_3 (t -2N\mu_n) }\exp\left(-N \left(\frac{2\pi^2 (\xi_1^2 + \xi_2^2) s_n^2}{n^2} +4\pi^2 \xi_3^2 v_n^2\right) \right)\\& 
 \notag\times \left( 1 + O\left(\frac{\|\xi\|_2^3}{n^3}+ |\xi_3|^3 (n^2 \log n)^3 \right)\right)^{2N}d\xi_3 +O_A(n^{-A}).
\end{align}
In the integral over $\xi_3$ substitute $\xi_3' = 2\pi N^{\frac{1}{2}} v_n \xi_3$ to obtain
\begin{align}
 &\Prob\left(S_N = (i,j,t)\right)  = \frac{1}{2\pi N^{\frac{1}{2}} v_n n^2} \sum_{\|\xi\|_2\ll \frac{n}{\sqrt{N}} \log n} e^{-2\pi i \left(\frac{\xi_1 i + \xi_2 j}{n} \right)}\\
 \notag
 &\times \int_{|\xi_3'| \ll \frac{v_n (\log n)^2}{n^2}} \exp\left(\frac{- i \xi_3' (t - 2N \mu_n)}{ N^{\frac{1}{2}} v_n} -{\xi_3'}^2 -N \left(\frac{2\pi^2 (\xi_1^2 + \xi_2^2)s_n^2}{n^2}\right)\right)\\
 \notag
 &\times \left(1 + O\left(\frac{\|\xi\|_2^3}{n^3} + |\xi_3'|^3 \left(\frac{n^2 \log n}{v_n \sqrt{N}} \right)^3\right)\right)^{2N} d\xi_3' + O_A(n^{-A}).
\end{align}
Completing the square, 
\begin{equation}
 \tilde{\xi}_3 = \xi_3' + \frac{i(t-2N\mu_n)}{2N^{\frac{1}{2}} v_n}
\end{equation}
then shifting the $\tilde{\xi_3}$ integral to be on the real axis obtains a horizontal integral bounded by $O_A(n^{-A})$ together with a shifted integral
\begin{align}
  &\Prob\left(S_N = (i,j,t)\right)  = \frac{\exp\left(-\frac{(t-2N\mu_n)^2}{4 N v_n^2} \right)}{2\pi N^{\frac{1}{2}} v_n n^2}\sum_{\|\xi\|_2\ll \frac{n}{\sqrt{N}} \log n} e^{-2\pi i \left(\frac{\xi_1 i + \xi_2 j}{n} \right)} \\ \notag
 &\times \int_{|\tilde{\xi}_3| \ll \frac{v_n (\log n)^2}{n^2}} \exp\left(- \tilde{\xi}_3^2 -N \left(\frac{2\pi^2 (\xi_1^2 + \xi_2^2)s_n^2}{n^2}\right)\right)\\ \notag
 &\times \left(1 + O\left(\frac{\|\xi\|_2^3}{n^3} + \left(|\tilde{\xi}_3|^3 + \frac{|t-2N\mu_n|^3}{N^{\frac{3}{2}}v_n^3}\right) \left(\frac{n^2 \log n}{v_n \sqrt{N}} \right)^3\right)\right)^{2N} d\tilde{\xi}_3 \\&\notag+ O_A(n^{-A}).
\end{align}
The main term is obtained by dropping the big $O$ terms and extending the $\tilde{\xi}_3$ integral to $\bR$ with acceptable error.  

To bound the error, note that in the region of integration, using $n^2 \ll v_n \ll n^2\log n$, 
\begin{equation}
 O\left(\frac{\|\xi\|_2^3}{n^3} + \left(|\tilde{\xi}_3|^3 + \frac{|t-2N\mu_n|^3}{N^{\frac{3}{2}}v_n^3}\right) \left(\frac{n^2 \log n}{v_n \sqrt{N}} \right)^3\right) = o\left(\frac{1}{N}\right),
\end{equation}
so that the exponential may be bounded linearly.  Bound integration over $\tilde{\xi}_3$ by a constant.  This obtains an error of
\begin{align}
 &\ll \frac{\exp\left(-\frac{(t-2N\mu_n)^2}{4 N v_n^2} \right)}{2\pi N^{\frac{1}{2}} v_n n^2} \sum_{\|\xi\|_2 \ll \frac{n}{\sqrt{N}} \log n} \exp\left(\frac{-2\pi^2 N \|\xi\|_2^2 s_n^2}{n^2} \right) \\ \notag&\times N\left(\frac{\|\xi\|_2^3}{n^3} + \left( 1 + \frac{|t - 2N\mu_n|^3}{N^{\frac{3}{2}} v_n^3}\right)\left(\frac{n^2 \log n}{v_n \sqrt{N}} \right)^3 \right).
\end{align}
Use $v_n \gg n^2$, and use $|t-2N\mu_n| \ll N^{\frac{1}{2}} v_n \log n$, and approximate the sum over $\xi$ with an integral over $\bR^2$ to estimate the error by
\begin{align}
&\ll \frac{\exp\left(-\frac{(t-2N\mu_n)^2 }{4 N v_n^2} \right)N^{\frac{1}{2}}}{2\pi  v_n }\int_{\bR^2} \exp\left(-2\pi^2 N s_n^2 x^2  \right) \left(\|x\|_2^3 + \frac{(\log n)^6}{N^{\frac{3}{2}}} \right)dx.
\end{align}
This obtains the claimed error bound.

\end{proof}

We can now prove Theorem \ref{single_piece_mixing_theorem}. The reader is referred to Appendix \ref{concentration_appendix} for the variant of Chernoff's inequality used here, which treats variables that have exponentially decaying tails.
\begin{proof}[Proof of Theorem \ref{single_piece_mixing_theorem}]
Recall that the time $t$ Brownian motion $B(t)$ on $(\bR/\zed)^2$ has a distribution which is a $\theta$ function $\theta_t(x)$.  The convergence in Theorem \ref{single_piece_mixing_theorem} consists of a lower bound and an upper bound approximating the distance to uniformity of the single piece with  
\begin{align}
 \left\|\theta_t - \bU_{(\bR/\zed)^2}\right\|_{\TV} &= \frac{1}{2} \int_{(\bR/\zed)^2} |\theta_t(x)-1| dx\\
 \notag &= \int_{(\bR/\zed)^2} \one(\theta_t(x)>1) (\theta_t(x)-1) dx.
\end{align}

Let $\rho: \bR \to [0,1]$ be a smooth cut-off function which satisfies $\rho(x) = 0$ if $x \leq -1$ and $\rho(x) = 1$ if $x \geq 1$.  Let, for $\varepsilon > 0$, $\rho_\varepsilon(x) = \rho\left(\frac{x}{\varepsilon} \right)$ and define $\psi_\varepsilon(x,t) = \rho_\varepsilon(\theta_t(x)-1)$.  For fixed $\varepsilon$ and for $t \in K$ with $K \subset \bR^+$ compact, $\psi_\varepsilon(x,t)$ is uniformly $C^1$ since $\theta_t(x)$ is uniformly $C^j$ for every $j$.  Also,
\begin{equation}
 \left|\left\|\theta_t(x) - \bU_{(\bR/\zed)^2}\right\|_{\TV} - \int_{(\bR/\zed)^2} \psi_\varepsilon(x,t) (\theta_t(x)-1)dx \right| \leq \varepsilon,
\end{equation}
since $|\theta_t(x)-1| \leq \varepsilon$ whereever $\psi_\varepsilon(x,t)$ and $\one(\theta(x,t)>1)$ differ.

 Define $c_{\puzzle} = \frac{2\mu}{\sigma^2} = \frac{5}{2}(\pi-1)$.  Let $(i_T, j_T)$ be the displacement from its initial position of the piece $\sP$ after $T = \lfloor c_{\puzzle} n^4 t \rfloor$ steps of the Markov chain $P$.  
 We show that for each fixed $\varepsilon > 0$, uniformly for $t \in K$,
 \begin{equation}\label{uniform_part}
  \lim_{n \to \infty} \E_{\bU_{(\zed/n\zed)^2}}\left[\psi_\varepsilon\left(\left(\frac{i}{n}, \frac{j}{n}\right), t \right) \right] = \int_{(\bR/\zed)^2} \psi_\varepsilon(x,t)dx
 \end{equation}
and 
\begin{equation}\label{random_part}
 \lim_{n \to \infty} \E_{P^T}\left[\psi_\varepsilon\left(\left(\frac{i_T}{n}, \frac{j_T}{n}\right), t \right) \right] = \int_{(\bR/\zed)^2} \psi_\varepsilon(x,t) \theta_t\left(x\right)dx,
\end{equation}
which proves that the total variation distance of the single piece process is bounded below in the limit by that of Brownian motion. Note that (\ref{uniform_part}) holds since the expectation is a Riemann sum for the integral, so that the convergence holds by uniform convergence.
 
 The distribution of $\sP$ is determined after $N$ steps of the renewal process in Theorem \ref{llt}, so we now remove the stopping time implicit in the renewal process and include the moves $H_0, V_0$ prior to the renewal process beginning.
 
 Let, for $\epsilon > 0$, 
 \begin{equation}
  N = \left\lfloor \frac{T}{2 \mu_n} - n^{1 +\epsilon}\right \rfloor.
 \end{equation}
 Since $\mu_n$ is of order $n^2$ by Lemma \ref{asymptotics_lemma}, $N$ is of order $n^2$.
Let
\begin{equation}
 M = \left \lfloor \frac{T - S_{N,3}}{2\mu_n} - n^{\frac{1}{2} + 2\epsilon}\right \rfloor .
\end{equation}
Then outside a set $S_{\bad}$
\begin{equation}
 (i_T, j_T) = (H_0, V_0) + (S_{N,1}, S_{N,2}) + (S_{M,1}, S_{M,2}) + (E_1, E_2)
\end{equation}
where $E_1, E_2$ are the set of moves of the piece after time $t_1 + S_{N,3} + S_{M,3}$.  The condition for membership in $S_{\bad}$ is that either  $t_1 \geq n^2 (\log n)^3$, $|S_{N,3} - 2N \mu_n| \geq \sqrt{N} n^2 (\log n)^2$ or $|S_{M,3} - 2M \mu_n| \geq \sqrt{M} n^2 (\log n)^2$.  Since $\frac{t_1}{n^2 \log n}$, has an exponentially decaying tail, the first part of $S_{\bad}$ has measure $O_A(n^{-A})$.  Next, since $\frac{r_i}{n^2 \log n}$ and $\frac{s_i}{n^2\log n}$ have exponentially decaying tails, by the variant of Chernoff's inequality, Lemma \ref{Chernoff_variant_lemma}, the measure of $|S_{N,3} - 2N \mu_n| \geq \sqrt{N} n^2 (\log n)^2$ is $O_A(n^{-A})$.  Outside this set, $M$ is of order $n^{1 +\epsilon}$, and so in this set, $|S_{M,3} - 2M \mu_n| \geq \sqrt{M} n^2 (\log n)^2$ has measure $O_A(n^{-A})$.  Hence $S_{\bad}$ has probability, for any $A> 0$, $O_A(n^{-A})$, so can be ignored. 

Outside $S_{\bad}$, $M$ is of order $n^{1 + \epsilon}$, so that w.o.p. $S_{M,1}, S_{M,2} \ll n^{\frac{1}{2} + \epsilon}$ by Lemma \ref{Chernoff_variant_lemma} this time applied to $H_i$ and $V_i$, which have exponentially decaying tails.  Similarly, by excluding $S_{\bad}$, $T-S_{N,3} - S_{M,3} - t_1 = O(n^{\frac{5}{2} + 2\epsilon})$. It then follows by Chernoff's inequality for the sum of $r_i$ and $s_i$, that $\sP$ moves $O(n^{\frac{1}{2} + 3\epsilon})$ times in $(E_1, E_2)$.  Since $H_0$ and $V_0$ are bounded $\ll \log n$ w.o.p., it follows that w.o.p. \begin{equation}(i_T, j_T) = (S_{N,1}, S_{N,2}) + O(n^{\frac{1}{2}+3\epsilon}).\end{equation}   Since $\psi_{\varepsilon}(x,t)$ is uniformly $C^1$, it suffices to prove (\ref{random_part}) for $(i_T, j_T)$ replaced by $(S_{N,1}, S_{N,2})$.

For any fixed $i,j$, the error term in applying Theorem \ref{llt} to $S_N$ summed in $t$ is $O(n^{-3+\epsilon})$, and hence may be ignored.  Also, the sum over $\xi$ may be extended to all of $\zed^2$ with negligible error. This gives a main term of
\begin{equation}
 \frac{1}{n^2} \sum_{\substack{\xi \in \zed^2}}e^{-2\pi i\frac{\xi \cdot(i,j)}{n}} \exp\left(-\frac{2\pi^2 \|\xi\|_2^2 s_n^2 N}{n^2} \right) = \frac{1}{n^2} \theta_{\frac{s_n^2 N}{n^2}}\left(\frac{i}{n}, \frac{j}{n} \right). 
 \end{equation}
It follows that uniformly in $t$, as $n \to \infty$,
\begin{equation}
 \E_{P^T}\left[ \psi_\varepsilon\left(\left(\frac{i_T}{n}, \frac{j_T}{n}\right), t \right)\right]\sim \frac{1}{n^2} \sum_{(i,j) \in (\zed/n\zed)^2} \psi_\varepsilon \left(\frac{i}{n}, \frac{j}{n} \right)\theta_{\frac{s_n^2 N}{n^2}}\left(\frac{i}{n}, \frac{j}{n} \right). 
\end{equation}
Since $\theta_{\frac{s_n^2 N}{n^2}} \to \theta_t$ uniformly as $n \to \infty$, the claim now follows by uniform convergence.  This completes the proof of the lower bound.
 
To prove the upper bound, note that, conditioned on $S_{N,3}$, $S_{N,1}$ and $S_{N,2}$ are independent of $(H_0, V_0)$, $(S_{M,1}, S_{M,2})$ and $(E_1,E_2)$.  As in the lower bound, drop the error terms from the limit in Theorem \ref{llt}, since these contribute measure $o(1)$.  Denote by $\tilde{S}_N$ the main term.  Also, we may restrict attention to $\tilde{S}_{N,3}$ such that $|\tilde{S}_{N,3}-2N\mu_n| \leq A v_n$ for a constant $A$, since the remaining part has measure $o(1)$ as $A \to \infty$.    Since convolution with the remaining distributions can only decrease the total variation distance, as can removing the conditioning, it suffices to prove that, conditioned on any $\tilde{S}_{N,3}$ which differs from its mean by a bounded multiple of its variance, the distribution of $(\tilde{S}_{N,1}, \tilde{S}_{N,2})$ has distance from uniform bounded by the total variation distance of $\theta_t(x)$ to uniform. This in fact follows from the convergence of the Fourier series in Theorem \ref{llt}.  

\end{proof}

\begin{cor}\label{expectation_corollary}
 For $T = \lfloor c_{\puzzle}n^4 t \rfloor$, the expected number of pieces in their original position is \begin{equation}\E_{P_{n^2-1}^T}[\#\{(i_T,j_T) = (i_0, j_0)\}] = \theta_t(0) + o(1)\end{equation} as $n \to \infty$.
\end{cor}

\begin{proof}
 By linearity of expectation, it suffices to consider a single initial piece, say at position $(i_0, j_0)$, and to prove that 
 \begin{equation}
  \Prob\left((i_T, j_T) = (i_0, j_0)\right) = (1 + o(1))\frac{\theta_t(0)}{n^2}.
 \end{equation}
Use that, conditioned on $S_{N,3}$, $(S_{N,1}, S_{N,2})$ are independent of $(H_0, V_0),$ $(S_{M,1}, S_{M,2}),$ $(E_1, E_2)$, which, together, have a distribution bounded with overwhelming probability by $n^{\frac{1}{2}+\epsilon}$.  Since the density function of $(S_{N,1}, S_{N,2})$  varies on a scale of order $n$, the claim now follows from Theorem \ref{llt} for $(S_{N,1}, S_{N,2})$.
\end{proof}

\section{A coupling of several labeled pieces}\label{coupling_section}

A coupling of a Markov process $(\sX, P)$ is a Markov process defined on the product $\sX \times \sX$ whose marginals are separately $P$.  Given a coupling $\{(X_t, Y_t)\}$, the coupling is said to coalesce if it has the property that if $X_s = Y_s$ then $X_t = Y_t$ for all $t>s$. The following theorem states that the coupling time bounds the total variation mixing time.
\begin{theorem}[\cite{LPW09}, Theorem 5.2]
 Let $\{(X_t, Y_t)\}$ be a coupling of Markov chains which coalesce.  Let $X_0 = x$ and $Y_0 = y$ be possibly random initial states.  Let $\tau_c$ be the first time that the chains coincide.  Then
 \begin{equation}
  \|P^t(x, \cdot) - P^t(y, \cdot)\|_{\TV} \leq \Prob(\tau_c > t).
 \end{equation}

\end{theorem}
We now give the proof of Theorem \ref{coupling_theorem}.

\begin{proof}[Proof of Theorem \ref{coupling_theorem}]
 Assume for convenience that $n$ is odd.  Let $\sP_1, ..., \sP_d$ be the labeled pieces of the first process with empty square $\sP_e$, and let the corresponding pieces be $\sP_1', ..., \sP_d'$ and $\sP_e'$ on the second copy of $\sX= \sX_d$. The standard coupling of $\sP_e$ and $\sP_e'$ on $(\zed/n\zed)^2$ has $\sP_e$ and $\sP_e'$ move in opposite directions until either the horizontal and vertical components coincide.  Once this occurs, the direction which coincides mirrors, while the other direction runs in the opposite direction until both coincide.  The expected time to coincide is of the order of the hitting time to 0 in simple random walk on $\zed/n\zed$, which is order $n^2$.
 
 The coupling on $\sX_d$  proceeds in the following steps.
 \begin{enumerate}
  \item Perform the standard coupling of $\sP_e, \sP_e'$ on $(\zed/n\zed)^2$ until they coincide.
  \item Move $\sP_e$ and $\sP_e'$ together until two pieces $\sP_i, \sP_i'$ have distance $O(1)$ and their center of mass in $(\zed/n\zed)$ is at distance $O(1)$ from them.
  \item Let $\sP_e, \sP_e'$ move independently until their center of mass is the same as that of $\sP_i$ and $\sP_i'$.
  \item Move $\sP_e, \sP_e'$ in opposite directions, preserving their center of mass, until $\sP_i$ and $\sP_i'$ coincide.
  
  \item [*]If either step 3 or 4 takes more than $O(1)$ moves, skip to step 5.
  \item Repeat while there are pieces whose positions do not coincide.
 \end{enumerate}
As already mentioned, the expected time for step (1) is $O(n^2)$.  The expected time for $\sP_i$ and $\sP_i'$ to have distance $O(1)$ can be estimated as follows.  By considering the hitting time of a reversible Markov chain on $(\zed/n\zed)^2$, see Lemma \ref{general_random_walk_hitting_time}, $\sP_i$ and $\sP_i'$ must move $O(n^2 \log n)$ times before their distance is $O(1)$.  Once this is the case, there is a positive probability that order 1 moves of either piece cause their center of mass also to be at distance $O(1)$ from them.  The expected time between moves is $O(n^2)$.  Hence step (2) takes time $O(n^4 \log n)$ in expectation. 

Notice that, while $\sP_e$ and $\sP_e'$ move together in step (2), they do not cause any pieces which already coincide to separate.  Hence pieces may separate only in steps (1), (3) and (4).  Note that, following step 4, $\sP_e$ and $\sP_e'$ have distance from each other which is $O(1)$, so they meet in $O(1)$ steps with probability bounded away from 0.  In particular, there is a probability bounded below such that steps (3), (4) and the next step of (1) are successful in $O(1)$ moves.  It follows that by repeatedly running the algorithm, it completes in $O(n^4 \log n)$ moves with probability at least $\frac{1}{2}$.
\end{proof}

\section{The mixing of several pieces, Proof of Theorem \ref{poisson_theorem}}

The main result of this section improves the mixing time estimate for the mixing of $d$ squares and the empty square in $\sX_d$ to $O(n^4)$.  Denote the $d_2$ norm on $\sX_d$ by
\begin{equation}
 \|f-g\|_{d_2}^2 = |\sX_d| \left(\sum_{x \in \sX_d} (f(x)-g(x))^2 \right).
\end{equation}

 Let $P_d$ be the transition matrix of the $n^2-1$ puzzle in $\sX_d$.  
\begin{theorem}\label{d_pieces_theorem}
 The $d_2$ mixing time of $d$ marked squares together with the empty square in the $n^2-1$ puzzle is $O(n^4)$.
\end{theorem}

Prior to giving the proof of Theorem \ref{d_pieces_theorem} we show how the Theorem implies the Poisson convergence in Theorem \ref{poisson_theorem}.  We first note the following spectral interpretation of the $d_2$ distance to uniformity.
\begin{lemma}\label{d2_lemma}
 Let $\lambda_0 = 1 > \lambda_1 \geq \cdots \geq \lambda_{|\sX_d|-1}$ be the eigenvalues of $P_d$ on $\sX_d$.  For any initial state $y$,
 \begin{equation}
  \|e_{y}^t P_d^N - \bU_{\sX_d}\|_{d_2}^2 = \sum_{i = 1}^{|\sX_d|-1} \lambda_{i}^{2N}.
 \end{equation}
\end{lemma}
\begin{rem}
 Since $P_d$ is $\frac{1}{5}$-lazy, $\lambda_{|\sX_d|-1} \geq -\frac{4}{5}$.
\end{rem}

\begin{proof}
 Let $\{v_i\}_i$ be an orthonormal basis of eigenvectors corresponding to $\{\lambda_i\}_i$.  We have, using that $v_0 = \frac{1}{\sqrt{|\sX_d|}} \sum_{x \in \sX_d} e_{x}$,
 \begin{equation}
  |\sX_d| \sum_{x \in \sX_d} \left|e_{y}^t P_d^N e_{x} - \frac{1}{|\sX_d|}\right|^2 = |\sX_d| \sum_{x \in \sX_d} \left(\sum_{i=1}^{|\sX_d|-1} \lambda_i^N \langle v_i, e_{x}\rangle \langle v_i, e_{y}\rangle\right)^2.
 \end{equation}
Since the permutation group of the $n^2-1$ puzzle acts transitively on $\sX_d$, which is a representation space for a random walk on a group, this is furthermore equal to 
\begin{align}
  \|e_{y}^t P_d^N - \bU_{\sX_d}\|_{d_2}^2 & = \sum_{x,y \in \sX_d} \left(\sum_{i=1}^{|\sX_d|-1} \lambda_i^N \langle v_i, e_{x}\rangle \langle v_i, e_{y}\rangle \right)^2\\
  &\notag = \sum_{i,j=1}^{|\sX_d|-1} \lambda_i^N \lambda_j^N \sum_{x, y \in \sX_d} \langle v_i, e_{x} \rangle \langle v_j, e_{x}\rangle \langle v_i, e_{y}\rangle \langle v_j, e_{y}\rangle
  \\& \notag =
  \sum_{i=1}^{|\sX_d|-1} \lambda_i^{2N}
\end{align}
in which, in the last line, we have used the orthogonality of the basis $\{e_{x}\}_{x \in \sX_d}$.
\end{proof}
We now deduce the Poisson convergence in Theorem \ref{poisson_theorem} from the mixing time bound in Theorem \ref{d_pieces_theorem}.
\begin{proof}[Deduction of Theorem \ref{poisson_theorem}]
 Let $\pi(i)$ denote the position of piece $i$.  To prove the convergence in the case $f(n) \to \infty$ with $n$, since the Poisson distribution is determined by its moments, it suffices that, for each fixed $d$, for $N = n^4f(n)$, as $n \to \infty$,
\begin{equation}
 E_d = \sum_{1 \leq a_1 < a_2 < ... < a_d < n^2} \E_{e_{\id}^t P_{n^2-1}^N} [\pi(a_1) = a_1 \wedge \cdots \wedge \pi(a_d) = a_d] \to \frac{1}{d!}.
\end{equation}

Let $a_{x_1, ..., x_d, x_0}$ be the indicator function on $\sX_d$ that $\sP_1$ is in position $x_1$, ..., $\sP_d$ is in position $x_d$ and $\sP_e$ is in position $x_0$, and let $b_{x_1, ..., x_d}$ be the indicator function that $\sP_1$ is in position $x_1$, ..., $\sP_d$ is in position $x_d$, neglecting the empty square.  The expectation is the partial trace
\begin{equation}
E_d= \frac{1}{d!} \sum_{(x_1, ..., x_d, (n,n)) \in \sX_d}a_{x_1, ..., x_d, (n,n)}^t P_d^N b_{x_1, ..., x_d}. 
\end{equation}
Let $\lambda_0 = 1 > \lambda_1 \geq \cdots \geq \lambda_{|\sX_d|-1} \geq -\frac{4}{5}$ be the eigenvalues of $P_d$ with orthonormal basis of eigenfunctions $v_0, ..., v_{|\sX_d|-1}$. Thus the partial trace is equal to
\begin{equation}
E_d= \frac{1}{d!} \sum_{i=0}^{|\sX_d|-1} \lambda_i^N \sum_{(x_1, ..., x_d, (n,n)) \in \sX_d} \langle a_{x_1, ..., x_d, (n,n)}, v_i\rangle \langle b_{x_1, ..., x_d}, v_i \rangle.
\end{equation}

The contribution of the constant eigenfunction $v_0$ is $\frac{1}{d!}$, so it suffices to bound the $i \neq 0$ terms.
To do so, first smooth out the $a_{x_1, ..., x_d, (n,n)}$ indicator function by convolving with $P_d$ $n^2$ times, which has the effect of randomizing the position of the empty square.  Since $P_d$ is self-adjoint, this obtains
\begin{align}
 &E_d - \frac{1}{d!} \\\notag &= \frac{1}{d!} \sum_{i=1}^{|\sX_d|-1} \lambda_i^{N-n^2} \sum_{(x_1, ..., x_d, (n,n) \in \sX_d} \langle P_d^{n^2} a_{x_1, ..., x_d, (n,n)}, v_i\rangle \langle b_{x_1, ..., x_d}, v_i\rangle.
\end{align}
Bound, using Cauchy-Schwarz,
\begin{align}
& \left|\sum_{(x_1, ..., x_d, (n,n)) \in \sX_d} \langle P_d^{n^2} a_{x_1, ..., x_d, (n,n)}, v_i\rangle \langle b_{x_1, ..., x_d}, v_i\rangle\right|\\
\notag& \leq \left(\sum_{(x_1, ..., x_d, (n,n)) \in \sX_d} \langle P_d^{n^2} a_{x_1, ..., x_d, (n,n)}, v_i\rangle^2 \right)^{\frac{1}{2}}\\
\notag& \times  \left(\sum_{(x_1, ..., x_d, (n,n)) \in \sX_d} \langle b_{x_1, ..., x_d}, v_i\rangle^2 \right)^{\frac{1}{2}}.
\end{align}
Since the $b$ functions are orthogonal and satisfy $\|b_{x_1, ..., x_d}\|_2^2 = n^2 - d$ (there are $n^2-d$ choices for the empty square), the second factor is of order $n$.
To handle the first factor, write $p(x, y)$ for the probability under $P_d^{n^2}$ of transitioning from $(x_1, ..., x_d, (n,n))$ to $(y_1, ..., y_d, y_0)$.  Since the empty square is performing simple random walk, and the sup of its distribution started from $(n,n)$ after $n^2$ steps is order $\frac{1}{n^2}$, applying Cauchy-Schwarz a second time, then orthogonality,
\begin{align}
 &\sum_{(x_1, ..., x_d, (n,n)) \in \sX_d} \left\langle \sum_{(y_1, ...,y_d, y_0) \in \sX_d} p(x,y) a_{y}, v_i\right\rangle^2 \\
 \notag &\leq \sum_{(x_1, ..., x_d, (n,n)) \in \sX_d} \sum_{y \in \sX_d} p(x, y) \langle a_{y}, v_i \rangle^2\\
 \notag &\ll \frac{1}{n^2} \sum_{y \in \sX_d} \langle a_{y}, v_i \rangle^2 \ll \frac{1}{n^2}.
\end{align}
The inequality in the last line holds since
\begin{equation}
 \sum_{(x_1, ..., x_d, (n,n)) \in \sX_d} p(x, y) \ll \frac{1}{n^2}
\end{equation}
is the probability of the empty square transitioning from $(n,n)$ to $y_0$ ignoring the first $d$ coordinates.

It follows that
\begin{equation}
 \left|E_d - \frac{1}{d!}\right| \ll \frac{1}{d!} \sum_{i=1}^{|\sX_d|-1} |\lambda_i|^{N-n^2}.
\end{equation}
If $\frac{N}{n^4} \to \infty$ as $n \to \infty$, this error term tends to 0 as $n \to \infty$ by Theorem \ref{d_pieces_theorem}, while, in any case, it remains bounded.

Since Corollary \ref{expectation_corollary} proves that the mean does not converge to 1 if $f(n)$ remains bounded,  boundedness of the second moment in this case guarantees that there is not convergence to $\Pois(1)$.  

Notice, also, that the claimed convergence in distribution for $f(n) \to \infty$ proves that, for any fixed $\epsilon$ the total variation distance from stationarity is bounded by $\epsilon$ in time $O(n^4)$.  Since the convergence to $\Pois(1)$ does not occur if $O(n^4)$ steps are taken, there is not cut-off, which proves Corollary \ref{no_cut-off_cor}.

\end{proof}

\subsection{Comparison and symmetrization} The remainder of this section gives the proof of Theorem \ref{d_pieces_theorem}.  This proof is performed by a sequence of comparisons to increasingly symmetric chains.  

Consider a symmetrized chain $P_{d,s}$ on the enlarged state space $\sX_{d,s}$ which has states in which the positions of $\sP_1, ..., \sP_d$ are as before, non-overlapping, but $\sP_e$ is allowed to take any position on $(\zed/n\zed)^2$.  In addition, at most two of the $\sP_i$ may overlap if the empty square overlaps them as well.  The transitions are as follows.  If $\sP_e$ does not overlap any $\sP_i$ then it transitions to one of its neighbors with probability $\frac{1}{10}$ and holds with probability $\frac{3}{5}$.  If $\sP_e$ overlaps a single $\sP_i$ then either $\sP_i$ and $\sP_e$ move to a neighbor each with probability $\frac{1}{10}$, $\sP_e$ moves alone to any neighbor, each with probability $\frac{1}{10}$, or the position holds with the remaining probability. If $\sP_e$ overlaps $\sP_i$ and $\sP_j$, either $\sP_i$ or $\sP_j$ but not both can move with $\sP_e$ to a neighbor each with probability $\frac{1}{10}$, or the position holds with probability $\frac{1}{5}$. 

Aside from the reversibility, we note several further features of $\sX_{d,s}$, which will be of use.
\begin{itemize}
 \item The expected time of $\sP_e$ at a square which it occupies alone is $\frac{5}{2}$, while the expected time at a square which it occupies with another $\sP_i$ is $\frac{5}{4}$. 
 \item Conditioned to move at a given step, $\sP_e$ performs simple random walk on $(\zed/n\zed)^2$.
 \item Conditioned to move at a given step, each $\sP_i$ performs simple random walk on $(\zed/n\zed)^2$.
\end{itemize}

  Let the  Dirichlet form of $P_{d,s}$ be $\sE_{d,s}$.  A function $f$ on $\sX_d$ may be extended to $\sX_{d,s}$ by requiring that if $x$ is a position at which $\sP_e$  and two $\sP_i, \sP_j$ overlap, the value of $f$ is equal to the average value of $f$ at the neighbors.  If $\sP_e$ and only one $\sP_i$ overlap, then the value of $f$ is taken to be the average value found by moving only $\sP_i$ to a neighbor.    Let the restricted Dirichlet form be $\sE_{d,s}'$. 
  \begin{lemma}\label{first_comparison_lemma}
   There is a fixed constant $C$, such that for any function $f$ on $\sX_d$,
\begin{equation}
 \sE_{d,s}'(f,f) \leq \frac{C}{|\sX_d|}\sum_{x \neq y \in \sX_d, x \sim y} (f(x)-f(y))^2 
\end{equation}
where $x \sim y$ means that $x$ and $y$ differ in only a neighborhood of diameter $O(1)$ about the empty piece in $x$. Furthermore, there is a constant $A_d$ such that
\begin{equation}
 \sE_{d,s}'(f,f) \leq A_d \sE_d(f,f)
\end{equation}
where $\sE_d$ is the Dirichlet form of the original chain $P_d$.  

In particular, if $1 = \lambda_0 > \lambda_1 \geq \cdots \geq \lambda_{|\sX_d|-1}$ are the eigenvalues of $P_d$ and $1 = \lambda_{0,s} > \lambda_{1,s} \geq \cdots \geq \lambda_{|\sX_{d,s}|-1}$ are the eigenvalues of $P_{s,d}$ then there is a constant $B_d>0$ such that $1-\lambda_{i,s} \leq B_d (1-\lambda_i)$. 
  \end{lemma}
\begin{proof}
We have $\frac{|\sX_{d,s}|}{|\sX_d|} \asymp 1$, so the  factor from the stationary measure may be replaced with $\frac{1}{|\sX_d|}$.

 If $x \in \sX_{d,s} \setminus \sX_d$ then there is a neighborhood $N_{x}$ of $x$, in which the points differ from $x$ only in a bounded neighborhood of the empty square in $x$, and convex weights $c_{x,y}$ for $y \in N_{x} \cap \sX_d$ such that $f(x) = \sum_{y \in N_{x} \cap \sX_d} c_{x, y} f(y)$. By Cauchy-Schwarz, 
 \begin{align}
  (f(x)-f(z))^2 &= \left(\sum_{y \in N_{x} \cap \sX_d} c_{x, y} (f(y) - f(z)) \right)^2\\
  \notag &\leq \sum_{y \in N_{x} \cap \sX_d} c_{x, y}(f(y)-f(z))^2.
 \end{align}
If $z \not \in \sX_d$ then the process can be repeated again so that differences of squares at points only in $\sX_d$ are obtained.  Since only bounded neighborhoods are considered, the first claim is obtained.  

Since if $x, y\in \sX_d$ differ only at points of bounded distance from the empty square in $x$ then there is a bounded length path under the Markov kernel $P_d$ from $x$ to $y$. Since the transition kernels of both chains are bounded from above and below by non-zero constants, the stationary measures are equal up to constants, and the paths are bounded length, it follows that $\sE_{d,s}'(f,f) \leq A_d \sE_d(f,f)$ by Theorem \ref{path_theorem}.

To prove the second claim, by the minimax characterization of eigenvalues
\begin{align}
 &1-\lambda_{i,s} \\ \notag & = \max\left\{ \min \frac{\sE_{d,s}(f,f)}{\|f\|_{\sX_{d,s}}^2}, f \in W\setminus \{0\}: W<\ell^2(\sX_{d,s}), \dim(W^\perp) = i\right\}\\
 \notag &\leq \max\left\{ \min \frac{\sE_{d,s}'(f,f)}{\|f\|_{\sX_d}^2}, f \in W \setminus\{0\}: W<\ell^2(\sX_{d}), \dim(W^\perp) = i \right\}\\
 \notag & \ll \max\left\{\min \frac{\sE_d(f,f)}{\|f\|_{\sX_d}^2}, f \in W \setminus \{0\}: W<\ell^2(\sX_d), \dim(W^\perp) =i\right\}\\
 &\notag= 1-\lambda_i.
\end{align}

\end{proof}

The spectral estimate of $(\sX_{d,s}, P_{d,s})$ is performed as follows. 
\begin{definition}
Let $M \geq 1$ be a parameter. Say that a configuration in $\sX_{d,s}$ is $M$-well-spaced if the labeled pieces $\sP_i$ and the empty square are spaced by distance at least $M$ in the $\ell^\infty$ metric. We use the abbreviations $\ws$ and $\nws$ for well-spaced and not-well-spaced.
\end{definition}
The space $\sX_{d,s}$ is split as the sum $\sX_{d,\ws} \oplus \sX_{d, \nws}$ into sets of labeled points which are $M$-well-spaced for some parameter $M$, or are  $M$-not-well-spaced.  For another parameter $M_1$, the operator $P_{d,s}^{M_1n^2}$ is shown to act smoothly and approximately independently on the well-spaced pieces, mapping the pieces to a neighborhood on a scale $\sqrt{M_1}$ independently, up to constant factors. The operator $P_{d,s}^{M_2 n^2}$ is shown to map essentially all of the mass from not-well-spaced points to well-spaced points.  Two cases are now considered.  Given an eigenfunction $f$, $\|f\|_2 = 1$, write $f = f_{\ws} + f_{\nws}$ as a sum of well-spaced and not-well-spaced points.  If $\|f_{\ws}\|_2 \gg 1$ then estimate
\begin{equation}
 \sE_{P_{d,s}^{M_1 n^2}}(f,f) \geq \frac{1}{2|\sX_{d,s}|}\sum_{x, y \ws} (f(x)-f(y))^2 P_{d,s}^{M_1n^2}(x,y).
\end{equation}
This Dirichlet form is then compared to a further symmetrized chain.
If $\|f_{\nws}\|_2 \geq 1-\epsilon$, use 
\begin{equation}
\left \langle \left(I-P_{d,s}^{M_2 n^2}\right) f, f \right\rangle \geq \left\langle \left(I - P_{d,s}^{M_2 n^2} \right) f_{\nws}, f_{\nws}\right\rangle + O(\|f_{\ws}\|_2).
\end{equation}
Since $P_{d,s}^{M_2n^2}$ moves most of the mass away from not-well-spaced points, this is bounded below by a constant.

In order to prove these estimates we first demonstrate some concentration estimates for simple random walk.

\subsubsection{Concentration estimates}
We use several concentration estimates relating the number of times the empty square and labeled pieces move and steps of the Markov chain. The exponential moment bounds used in these estimates are recorded in Appendix \ref{concentration_appendix}.

\begin{lemma}\label{empty_moves_concentration}
 Let $E_{N,i}$ be the number of times the empty square moves under $P_{d,s}$, not counting holds, up to the $N$th move of piece $i$.  There is a constant $c > 0$ such that, for $\lambda>1$, for $N\geq 1$,
 \begin{equation}
  \Prob\left(\left|E_{N,i} - \E[E_{N,i}] \right| > \lambda \sqrt{N} n^2\log n \right) \ll  e^{-c \lambda^2} +  e^{-c \lambda^{\frac{1}{2}}N^{\frac{1}{4}}}.
 \end{equation}

\end{lemma}
\begin{proof}
 After the initial move, consider the number of moves of the empty piece prior to the next move of piece $i$.  To describe this, first flip a fair coin, and with probability $\frac{1}{2}$ move piece $i$ together with $\sP_e$, so that the number of $\sP_e$ moves alone is 0.  With probability $\frac{1}{2}$ $\sP_e$ leaves piece $i$ and now performs a  return to piece $i$ under simple random walk.  Thus the number of moves of $\sP_e$ per move of piece $i$ is a sequence of random variables independent of the first number of moves, and with distribution that is i.i.d.  with expected number of moves of order $n^2$.  Both this variable, and the length of the initial move, rescaled by $n^2 \log n$, have exponentially decaying tail. The claim now follows by treating the error from the initial move and later ones separately and using the variant of Chernoff's inequality in Lemma \ref{Chernoff_variant_lemma}.
\end{proof}
Arguing as above, there is a constant $c_1 >0$ such that, as $N \to \infty$,
\begin{equation}
 \E[E_{N,i}] = c_1 n^2 N + O(n^2 \log n),
\end{equation}
in which $O(n^2 \log n)$ bounds the expected length of the first move. 
Concentration of the length of the renewal process implies concentration of the number of renewals as follows. 
\begin{lemma}\label{piece_i_concentration}
 Let $W_{M,i}$ be the number of times piece $i$ moves in $M\gg n^2(\log n)^{12}$ moves of the empty square.  There is a constant $c > 0$, such that, for $1<\lambda<(\log n)^2$,
 \begin{equation}
  \Prob\left(\left|c_1n^2 W_{M,i} - M \right| > \lambda \sqrt{M} n \log n \right) \ll e^{-c \lambda^2}.
 \end{equation}

\end{lemma}
\begin{proof}
The event $W_{M,i} \leq k$ (piece $i$ moves at most $k$ times in the first $M$ moves of the empty square) is contained in $E_{k+1, i} \geq M+1$ (the number of moves of the empty square when piece $i$ has moved $k+1$ times is at least $M+1$).  Similarly, $W_{M, i} \geq k$ is contained in $E_{k,i} \leq M$.  Thus
\begin{align}
 &\Prob\left(|c_1 n^2 W_{M,i} - M| > \lambda\sqrt{M} n \log n \right)\\
 \notag &\leq  \Prob\left(E_{k+1,i}\geq M+1: k = \frac{M}{c_1n^2} - \frac{\lambda \sqrt{M}}{c_1n} \log n\right)\\ \notag&+  \Prob\left(E_{k,i} \leq M: k = \frac{M}{c_1n^2} + \frac{\lambda \sqrt{M}}{c_1n} \log n \right).
\end{align}
We have $\E[E_{k,i}] = c_1 n^2 k + O(n^2 \log n)$ and thus in the first case
\begin{equation}
 M+1 - \E[E_{k+1,i}] = \lambda \sqrt{M} n \log n + O(n^2 \log n)
\end{equation}
while in the second case
\begin{equation}
 \E[E_{k,i}] - M = \lambda\sqrt{M} n \log n + O(n^2 \log n).
\end{equation}
Thus, $\lambda$ in the previous lemma can be taken to within constants in this lemma.  Since $\lambda = O((\log n)^2)$ while $k \asymp \frac{M}{n^2}$, of the two bounds $\lambda^2$ and $\lambda^{\frac{1}{2}} k^{\frac{1}{4}}$, the first is smaller.
\end{proof}
The following lemma bounds the maximum displacement of piece $i$ as it moves.
\begin{lemma}\label{displacement_concentration}
 Let $X_{i,1}, X_{i,2}, X_{i,3}, ...$ be the sequence of moves of piece $i$ and let 
 \begin{equation}
  \Sigma_{i,n} = \sum_{j=1}^n X_{i,j}
 \end{equation}
be its partial sum.  We have, for $T < n$, for some $c>0$, for $\lambda > 1$,
\begin{equation}
 \Prob\left(\max_{1 \leq j \leq T} \left|\Sigma_{i,j}\right|> \lambda \sqrt{T} \right) \ll e^{-c \lambda^2}.
\end{equation}
In particular, it holds with overwhelming probability that for any $1 \leq a < b \leq n^3$ that $|\Sigma_{i,b} - \Sigma_{i,a}| \leq \sqrt{b-a} \log n$.
\end{lemma}
\begin{proof}
 The distribution of the moves of piece $i$ are those of simple random walk.  The claim now follows from Lemma \ref{sup_bound_lemma} applied separately to the $x$ and $y$ coordinates.
 
 The last claim holds by making a union bound.
\end{proof}

\begin{lemma}\label{empty_square_concentration}
 Let $K_N$ be the number of transitions of $P_{d,s}$ when the empty square has moved $N$ times.  For $ n^2(\log n)^{12}\ll N \ll n^3$, for $1<\lambda \leq (\log n)^2$, for some $c>0$,
 \begin{equation}
 \Prob\left(\left|K_N - \frac{5}{2}N\right| > \lambda \sqrt{N} \right) \ll e^{-c\lambda^2}.
 \end{equation}
\end{lemma}

\begin{proof}
 By Lemma \ref{piece_i_concentration}, piece $i$ moves at most $\ll \frac{N}{n^2}$ times with probability $1 - O(e^{-c\lambda^2})$.    It follows that we may assume that the empty piece is moving from a position that it occupies alone on at least $N - O(\frac{N}{n^2})$ of its moves, by excluding a set of probability having the required measure.
 
 The wait between moves when the empty square is on a position alone is a geometric random variable of mean $\frac{5}{2}$.  When it occupies a position with another piece it is a geometric variable of mean $\frac{5}{4}$. The number of transitions when the empty square overlaps another piece may be absorbed into the claimed error, since with the claimed probability it is bounded by $O\left(\frac{N(\log n)^2}{n^2} \right)$.  By the variant of Chernoff's inequality for variables of exponentially decaying tails (Lemma \ref{Chernoff_variant_lemma}) the claim follows, since $\lambda^2 \leq \lambda^{\frac{1}{2}} N^{\frac{1}{4}}$.
\end{proof}




Define a stopping time 
\begin{equation}
 T_m = \min\{k \geq 1: \text{ at step $k$ } \sP_e \text{ has moved $m$ times}\}.
\end{equation}
\begin{lemma}\label{occupation_time_concentration}
 The following holds w.o.p.  
 
 Let $(\log n)^{4} < K < n$.  For any $t \geq 0$, given any initial position $x_t \in \sX_{d,s}$, for any $(i,j)$ at distance $> \sqrt{K}\log n$ from each labeled piece in $x_t$, the number of steps of the chain which the empty square spends at $(i,j)$  between $T_t$ and $T_{t + Kn^2}$ is $\frac{5}{2}K + O\left(K^{\frac{1}{2}}(\log n)^{2} \right).$
\end{lemma}
\begin{proof}
 Sampled at the times $T_t$, $T_{t+1}, T_{t+2}, ...$ the empty square is performing simple random walk.  With probability $1 + O_A(n^{-A})$, the first visit of the stopped process to $(i,j)$ is made in time $O(n^2 (\log n)^2)$. After the initial visit, the returns to $(i,j)$ of the stopped process have independent length, which is mean $n^2$ and variance $O(n^4 (\log n)^2)$, and has exponentially decaying tail when rescaled by $n^2 \log n$.  It follows that by the variant of Chernoff's inequality in Lemma \ref{Chernoff_variant_lemma} that, w.o.p. between $T_t$ and $T_{t + Kn^2}$ the number of visits of the stopped process to $(i,j)$ is $K + O(K^{\frac{1}{2}} (\log n)^2)$.  
 
 By Lemma \ref{piece_i_concentration}, the number of times that each labeled piece moves between $T_t$ and $T_{t + Kn^2}$ is w.o.p. bounded by a constant times $K$.  Also, w.o.p. the maximum displacement of any labeled piece is, by Lemma \ref{displacement_concentration}, $\ll \sqrt{K} \log n$. 
 
 On the empty piece's visits to position $(i,j)$, the holding time is a geometric random variable of mean $\frac{5}{2}$ if $(i,j)$ is unoccupied and of mean $\frac{5}{4}$ if the position is occupied by a labeled piece.  Keep separate lists of geometric random variables for unoccupied and occupied moves.  By the variant of Chernoff's inequality for geometric random variables, the sum of the first $K-CK^{\frac{1}{2}}(\log n)^2$ geometric random variables at $(i,j)$ when $(i,j)$ is unoccupied is at least $\frac{5}{2}K - O(K^{\frac{1}{2}}(\log n)^2)$ w.o.p. Similarly, the sum of the first $ K + C K^{\frac{1}{2}}(\log n)^2$ geometric random variables at $(i,j)$ when $(i,j)$ is unoccupied is $\frac{5}{2}K + O(K^{\frac{1}{2}}(\log n)^2)$ w.o.p. By a union bound, w.o.p. between $T_t$ and $T_{t + Kn^2}$, no $(i,j)$ which has distance at least $\sqrt{K}\log n$ from the position of a labeled piece in $x_t$ is occupied by a labeled piece during the time interval, and the number of transitions of the Markov chain with the empty square at $(i,j)$ is $\frac{5}{2}K + O(K^{\frac{1}{2}}(\log n)^2)$.

\end{proof}

\subsection{Spectral estimates}
Recall the heat kernel 
\begin{equation}
 H_t = e^{-t} \sum_{k=0}^\infty \frac{t^k P_{d,s}^k}{k!} = \sum_{\lambda \in \sigma(P_{d,s})} e^{t(\lambda-1)} v_\lambda v_\lambda^t.
\end{equation}

Let $\phi$ be a smooth bump function, $\phi \geq 0$, $\int \phi =1$, $\supp(\phi) \subset [1,2]$, and define
\begin{align}
 \Phi_t &= \int_1^2 \phi(s) H_{ts}ds,\\
 \notag \Phi_t &= \sum_{\lambda \in \sigma(P_{d,s})} v_\lambda v_\lambda^t \hat{\phi}(t(1-\lambda))
\end{align}
where $\hat{\phi}(z) = \int_1^2 \phi(s) e^{-sz}ds$ is the Laplace transform.  Write
\begin{equation}
 \Phi_t = \sum_k \omega_t\left(k\right)P_{d,s}^k.
\end{equation}

The definition of $\omega_t(k)$ may be extended to a function of a non-negative real variable $k$ via the integral representation,
 \begin{equation}
  \omega_t(k) = \int_1^2 \phi(s) \frac{e^{-st}(st)^k}{\Gamma(k+1)} ds.
 \end{equation}
In Appendix \ref{concentration_appendix} the following lemma is proved regarding the weights $\omega_t(k)$.

\begin{lemma}\label{functional_concentration}
 The function $\omega_t(k)$ satisfies $\omega_t(k) \ll \frac{1}{t}$ and the first discrete derivative satisfies
 \begin{equation}
  |D \omega_t(k)| \ll \frac{1}{t^2}.
 \end{equation}
For each fixed $\epsilon > 0$, for some $c>0$,
\begin{equation}
 1-\sum_{(1-\epsilon)t \leq k \leq (2+\epsilon)t} \omega_t(k) \ll e^{-c \epsilon t}.
\end{equation}

\end{lemma}
We also use the following distributional result.
\begin{lemma}\label{piece_move_lemma}
 Let $X_{i,m}$ be the displacement of piece $i$ after it has moved $m$ times.  The joint distribution of
 \begin{equation}
  X_m = (X_{1,m}, ..., X_{d,m})
 \end{equation}
 is the same as that of $d$ independent simple random walks on $(\zed/n\zed)^2$ of length $m$.
\end{lemma}
\begin{proof}
Introduce an auxilliary step in each step of the dynamics which chooses, with some probabilities, that the position holds or moves, and which pieces move.  With this intermediate step it is possible to stop the chain prior to the $j$th time that piece $i$ moves, which is then independent of all prior moves of the pieces and has the distribution of nearest neighbor simple random walk.  The claim in general now follows by induction, conditioning on the first piece that moves.
\end{proof}

The following lemma controls the part of the spectrum localized near not-well-spaced points.
\begin{lemma}\label{not_well_spaced_lemma}
 Let $M = \exp(\sqrt{\log n})$ and let $t = M^{5}n^2$.  If $\|f\|_2 = 1$ and $f$ is supported on $M$-not-well-spaced points, then $|f^t \Phi_t f| < \frac{1}{2}$.
\end{lemma}
\begin{proof}
By the concentration estimate in Lemma \ref{functional_concentration}, for some fixed $\epsilon>0$,
\begin{equation}
 \Phi_t = \sum_{(1-\epsilon)t \leq k \leq (2+\epsilon) t} \omega_t(k) P_{d,s}^k + E
\end{equation}
where $\|E\|_{2\to 2}$ is $O(e^{-c\epsilon t})$.  By Lemma \ref{empty_square_concentration} the total number of times the empty square moves is w.o.p. between $(1-\epsilon)\frac{2}{5}t $ and $ (2+\epsilon)\frac{2}{5}t$.

 Let $K = M^{\frac{5}{2}}(\log n)^2$ and let $m_j = (1-\epsilon)\frac{2}{5}t + j K n^2$ for $m_j \leq (2+\epsilon)\frac{2}{5}t$, and consider the stopping times $T_{m_j}$.  By Lemma \ref{piece_i_concentration}, w.o.p., for each $j$, the number of times $\sP_i$ moves up to time $T_{m_j}$, is
 \begin{equation}
  \left(\frac{(1-\epsilon)\frac{2}{5} t}{c_1 n^2} + \frac{jK}{c_1}\right) + O\left(M^{\frac{5}{2}} (\log n)^2 \right).
 \end{equation}
Consider the joint distribution, with $m = \frac{(1-\epsilon)\frac{2}{5} t}{c_1 n^2} + \frac{jK}{c_1}$,
\begin{equation}
 \sD_m^d = (X_{1, m}, X_{2,m}, ..., X_{d,m})
\end{equation}
where $X_{i,m}$ is the displacement from its initial position of $\sP_i$ after it has moved $m$ times.  By Lemma \ref{piece_move_lemma}, $\sD_m^d$ has the distribution of $d$ independent simple random walks on $(\zed/n\zed)^2$ of $m$ steps.   By the local limit theorem on $\zed^{2d}$, a Gaussian approximation to the density of $\sD_m^d$ holds for points at distance $\ll m^{\frac{1}{2} + \epsilon}$ from 0, outside of which the density has mass $O_A(n^{-A})$. 

By considering the distribution of simple random walk, see Lemma \ref{sup_bound_lemma}, the maximum displacement of piece $\sP_i$ during $O\left(M^{\frac{5}{2}} (\log n)^2 \right)$ moves is w.o.p. $
 O\left(M^{\frac{5}{4}} (\log n)^{2} \right).$
Between $T_{m_j}$ and $T_{m_{j+1}}$, w.o.p. each piece $i$ moves $\frac{K}{c_1} + O(\sqrt{K}\log n)$ times, and hence w.o.p.  the maximum displacement during this interval is $O\left(K^{\frac{1}{2}} \log n\right) = O\left(M^{\frac{5}{4}} (\log n)^{2}\right)$.

By Lemma \ref{occupation_time_concentration}, 
between $T_{m_j}$ and $T_{m_{j+1}}$ the number of times the empty square visits each point $(x,y) \in (\zed/n\zed)^2$ at distance at least $\sqrt{K}\log n$ from the positions of the labeled pieces at time $T_{m_j}$, is w.o.p., $K + O(K^{\frac{1}{2}} (\log n)^2)$.  Since $T_{m_j} = \frac{5}{2} m_j + O\left(m_j^{\frac{1}{2}} (\log n)^2\right)$ w.o.p., and thus, by the derivative estimate in Lemma \ref{functional_concentration}, for all $k\in [T_{m_j}, T_{m_{j+1}}]$, outside a set of measure $O_A(n^{-A})$, 
\begin{equation}
 \omega_t(k) = \omega_t\left(\frac{5}{2}m_j\right) + O\left(\frac{m_j^{\frac{1}{2}} (\log n)^2}{t^2} \right).
\end{equation}
Summed in $k \in [(1-\epsilon)t,(2 +\epsilon)t]$ the error introduced in $\Phi_t$ by this approximation is bounded by $O\left(\frac{\log n}{t} \right)$ in the $2\to 2$ operator norm.
Dropping this error, the contribution to $\Phi_t$ over this interval is asymptotic to the joint distribution of independent Gaussians times the uniform measure on $(\zed/n\zed)^2$ for the empty square outside a set in which the empty square is near a labeled piece, when the Gaussians are sampled on scales larger than $O\left(M^{\frac{5}{4}} (\log n)^{2} \right)$.

Given a fixed point $x$ which is not-well-spaced, the probability that $e_{x}^t \Phi_t$ is not-well-spaced is bounded by the sum of the probabilities that any pair of pieces is spaced by at most $M$, which in turn is bounded by the probability that any pair are spaced by at most $O\left(M^{\frac{5}{4}} (\log n)^{2} \right)$.  This can be estimated at the given scale, and is bounded by the ratio of the area of a circle of radius $M^{\frac{5}{4}}(\log n)^{2}$ to the area of a circle of radius $M^{\frac{5}{2}}$, which is $\ll M^{-\frac{5}{2} + \epsilon}$. 

By Cauchy-Schwarz,
\begin{align}
 |f^t \Phi_t f| &= \left|\sum_{x,y} f(x)f(y)\Phi_t(x,y)\right|\\ \notag
 & \leq \left(\sum_{x,y \nws} f(x)^2 \Phi_t(x,y)\right)^{\frac{1}{2}} \left(\sum_{x,y \nws} f(y)^2 \Phi_t(x,y)\right)^{\frac{1}{2}} \\& \notag \ll M^{-\frac{5}{2} + \epsilon}.
\end{align}

\end{proof}

The following lemma is used  for comparison in the spectrum concentrated on well-spaced points.

\begin{lemma}\label{well_spaced_lemma}
 Let $M = \exp(\sqrt{\log n})$ and $t = M^2 n^2$.  For any $C>0$ there is a $\delta(C)>0$ such that if $x, y$ are $M$-well-spaced and if the labeled pieces of $y$ are at distance at most $C M$ from those in $x$ then
 \begin{equation}
  \Phi_t(x, y) \geq \frac{\delta}{n^2 M^{2d}}.
 \end{equation}

\end{lemma}
\begin{proof}
It suffices to prove the claim for
\begin{equation}
 \tilde{\Psi}_t = \sum_{k \geq 0} \omega_t(T_k)P_{d,s}^{T_k}
\end{equation}
which satisfies $\tilde{\Psi}_t \leq \Phi_t$.  In fact we will prove the claim for 
\begin{equation}
 \Psi_t = \sum_{(1-\epsilon)\frac{2}{5}t \leq k \leq (2+\epsilon)\frac{2}{5} t} \omega_t\left(\frac{5}{2}k\right) P_{d,s}^{T_k}.
\end{equation}
Notice that, w.o.p., for all $(1-\epsilon)\frac{2}{5}t \leq k \leq (2+\epsilon)\frac{2}{5}t $, $T_k = \frac{5}{2}k + O(\sqrt{k} \log n)$, so that this may be assumed.  Under this assumption, the error in replacing $T_k$ with $\frac{5}{2}k$ in $\omega_t$ is $O\left(\frac{\sqrt{k}\log n}{t^2} \right)$.  Meanwhile, for each $(i,j)$, for each $k$ in the interval, the probability that the empty square has position $(i,j)$ in $P_{d,s}^{T_k}$
is $O\left(\frac{1}{n^2} \right)$, so that the error introduced in estimating $\Psi_t(x, y)$ is $O\left(\frac{\sqrt{k}\log n}{t n^2} \right)$, which is negligible compared to the claimed main term. Let $X_1, X_2, ...$ be the steps of the stopped Markov process started from $x$ and taking steps each time the empty square moves. Throughout this argument, for a vector $y$ of length $2d$, $y'$ indicates $y\times \{(0,0)\}$.

 The
proof is by an induction on scales.
 Let \begin{equation}M^2=K_1 > K_2 > \cdots > K_r >(\log n)^{21}\end{equation} be a sequence of scales, $K_{i+1}= K_i^{\frac{1}{2}}(\log n)^{10}$.  We have $r = O(\log\log n) .$
 We prove the following pair of statements, followed by a perturbation argument. 
 Let $M^{\frac{1}{2}}(\log n)^3 \leq R_1 \leq C M$. There is a constant $c>0$ such that, for all $y \in \zed^{2d}$, $\|y\|_2 \leq C M$, for all $(i,j) \in (\zed/n\zed)^2$,
 \begin{equation}\label{base_case}
  \Prob\left(e_{x}^t \Psi_t \in x + y' + [-R_1, R_1]^{2d} \times\{(i,j)\}\right) \geq c \frac{R_1^{2d}}{n^2 M^{2d}},
 \end{equation}
 and, for $1 < h \leq r$, for $K_h^{\frac{1}{4}}(\log n)^3 \leq  R_h \leq C M$,
 \begin{align}\label{inductive_step}
  &\Prob\left(e_{x}^t \Psi_t \in x + y' + [-R_{h}, R_{h}]^{2d} \times \{(i,j)\}\right) \\&\notag \geq c \left(1-O\left(\frac{h}{\log n} \right) \right) \frac{R_{h}^{2d}}{n^2 M^{2d}}.
 \end{align}
Note that $r \ll \log \log N$, so that the error term is $o(1)$.

 To start the induction and prove (\ref{base_case}), let \begin{equation}m_j = (1-\epsilon)t + j M n^2(\log n)^2,\qquad m_j \leq (2+\epsilon)t\end{equation} and argue as in the proof of Lemma \ref{not_well_spaced_lemma} to conclude that w.o.p., for each $j$ the number of moves of piece $i$ up to time $T_{m_j}$ is
 \begin{equation}
  \left(\frac{(1-\epsilon)\frac{2}{5}t}{c_1n^2} + \frac{jM(\log n)^2}{c_1} \right) + O(M(\log n)^2).
 \end{equation}
Let $m = \left(\frac{(1-\epsilon)\frac{2}{5}t}{c_1n^2} + \frac{jM (\log n)^2}{c_1} \right)$. With overwhelming probability, the distribution at $T_{m_j}$ differs from $\sD_m^d$ by $O(M^{\frac{1}{2}} (\log n)^{2})$ and this remains true of $T_k$ for $m_j \leq k \leq m_{j+1}$. Since at the times of the stopped process $T_{m_j}, T_{m_j+1}, ...$ the empty piece is performing simple random walk independent of the position at $T_{m_j}$, w.o.p. the number of visits to each position on $(\zed/n\zed)^2$ is $M (\log n)^2 + O(\sqrt{M}(\log n)^3)$.  In particular, for $T_{m_j} \leq k \leq T_{m_{j+1}}$, w.o.p. 
\begin{equation}
 \omega_t\left(\frac{5k}{2}\right) = \omega_t\left(\frac{5}{2}m_j \right) + O\left(\frac{M (\log n)^2}{t^2} \right).
\end{equation}
Summed in $(1-\epsilon)t \leq k \leq (2+\epsilon)t$, the error introduced in the weight is  $O\left(\frac{(\log n)^2}{t}\right)$.  Since for $(1-\epsilon)t \leq k \leq (2+\epsilon)t$, $e_{x}^tP^{T_k}e_{y} \ll \frac{1}{n^2}$ due to the equidistribution of the empty square, the error introduced in estimating $e_{x}^t \Psi_t e_{y}$ is $\ll \frac{(\log n)^2}{tn^2}$, which may be ignored. 

Since $\sD_m^d$ has an approximately joint independent Gaussian distribution with standard deviation of order $M$, for any $(i,j) \in (\zed/n\zed)^2$, for any $x$ which is $M$-well-spaced, for any $\|y\|_2 \leq C M$, for any $M^{\frac{1}{2}} (\log n)^3 \leq R \leq C M$,
\begin{equation}
 \Prob\left(e_{x}^t \Psi_t \in x+y' + [-R,R]^{2d} \times \{(i,j)\} \right) \asymp \frac{R^{2d}}{n^2M^{2d}}.
\end{equation}

For the inductive step, to prove (\ref{inductive_step}), assume that for all $K_{h-1}^{\frac{1}{4}}(\log n)^3 \leq R \leq C M$, for all $\|y\|_2 \leq C M$, for all $(i,j) \in (\zed/n\zed)^2$, 
\begin{align}
 &\Prob\left(e_{x}^t \Psi_t \in x + y' + [-R, R]^{2d} \times \{(i,j)\} \right) \\\notag & \geq c \left(1-O\left( \frac{h-1}{\log n}\right) \right) \frac{R^{2d}}{n^2 M^{2d}}.
\end{align}
 Let $(1-\epsilon)\frac{2}{5} t \leq m \leq (2+\epsilon)\frac{2}{5}t$.  Let $Y_{m,h}$ be the conditional change in the position of the pieces from stopping time $T_{m-K_h n^2}$ to $T_m$, conditioned on the position at $T_{m-K_h n^2}$. Conditioned on the position at $T_{m-K_h n^2}$, the number of times each piece $\sP_i$ moves between $T_{m-K_hn^2}$ and $T_m$ is, by Lemma \ref{piece_i_concentration},
 \begin{equation}
   \frac{K_h}{c_1} + O\left(K_h^{\frac{1}{2}}(\log n)^2 \right)
 \end{equation}
 w.o.p..
 It follows as before that the difference from the distribution of $\frac{K_h}{c_1}$ steps of independent simple random walk in each coordinate is $O\left(K_h^{\frac{1}{4}}(\log n)^2\right)$ w.o.p.. 
 Given $K_h^{\frac{1}{4}}(\log n)^3 \leq R_h \leq K_{h-1}^{\frac{1}{4}}(\log n)^3$, let $R_{h-1} = K_{h-1}^{\frac{1}{4}}(\log n)^3$ and let $S_{h-1}$ be a set of points in $\zed^{2d}$ such that the offsets $\{y + [-R_{h-1}, R_{h-1}]^{2d}: y \in S_{h-1}\}$ tile the $\sqrt{\frac{K_h}{c_1}}\log n$ neighborhood of 0 in $\zed^{2d}$.  Now calculate, conditioning on the location of $X_{k-K_hn^2}$,
 \begin{align}
  &\Prob\left(e_{x}^t \Psi_t \in x + y' + [-R_h, R_h]^{2d} \times \{(i,j)\}\right)\\ \notag
  &= \sum_k \omega_t\left(\frac{5k}{2}\right)\Prob\left( X_k \in x + y' + [-R_h, R_h]^{2d} \times \{(i,j)\} \right)\\
  \notag &\geq  \sum_{y_s \in S_{h-1}}\sum_k \omega_t\left(\frac{5k}{2}\right) \\\notag &\Prob\Bigl(X_{k} \in x + y' + [-R_h, R_h]^{2d} \times \{(i,j)\}\Big|\\ \notag &\qquad X_{k-K_h n^2} \in x + y' +y_s'  + [-R_{h-1}, R_{h-1}]^{2d} \times (\zed/n\zed)^2 \Bigr) \\\notag & \times \Prob\left(X_{k-K_h n^2} \in x + y'+ y_s'  + [-R_{h-1}, R_{h-1}]^{2d} \times (\zed/n\zed)^2  \right).
 \end{align}
For each fixed $(\ell, m) \in (\zed/n\zed)^2$ and for any $z \in y+y_s + [-R_{h-1}, R_{h-1}]^{2d}$, by the conditional distribution argument above,
\begin{align}
 &\Prob\Bigl(X_k \in x + y' + [-R_h, R_h]^{2d}\times\{(i,j)\}\Big| X_{k-K_hn^2} = x + z' \times\{(\ell, m)\}\Bigr)\\
 &\notag = O_A(n^{-A}) + \left(1 + O\left(\frac{1}{\log n}\right)\right) \frac{(2 R_h)^{2d}}{n^2 \left(2\pi \frac{K_h}{c_1} \right)^d}\exp\left(-\frac{\|y_s\|_2^2}{ \frac{2 K_h}{c_1}} \right).
\end{align}
The factor of $\left(1 + O\left(\frac{1}{\log n}\right)\right)$ accounts for two errors.  First, the distribution of $Y_{m,h}$ agrees with $\sD_m^d$ up to an error of size $O\left(\frac{R_h}{\log n} \right)$, and hence, adding or subtracting a boundary strip of length $O\left(\frac{R_h}{\log n} \right)$ to the rectangle $[R_h, R_h]^{2d}$ and dropping the error in the distribution gives an upper bound or lower bound for the measure.  Second, $z$ differs from $y+y_s$ by $O\left(R_{h-1} \right)$, while the points in the rectangle differ from $y$ by $O(R_h)$.  Also, $\|y_s\|_2 \ll \sqrt{\frac{K_h}{c_1}} \log n$.  Since \begin{equation}\exp\left( -\frac{\|x + \varepsilon\|_2^2}{2\sigma^2}\right) = \exp\left( -\frac{\|x\|_2^2}{2\sigma^2}\right) \left(1 + O\left( \frac{\|x\|_2\|\varepsilon\|_2 + \|\varepsilon\|_2^2}{\sigma^2}\right)\right)\end{equation} when $\|x\|_2\|\varepsilon\|_2, \|\varepsilon\|_2^2 \leq \sigma^2$, applying this with $\sigma^2 = \frac{K_h}{c_1}$, $x = y_s$ and $\|\varepsilon\|_2 = O(R_{h-1})$ obtains a relative error  of \begin{equation}1 + O\left(\frac{R_{h-1}\log n}{\sqrt{K_h}} \right) = 1 + O\left(\frac{K_{h-1}^{\frac{1}{4}}(\log n)^4}{K_{h-1}^{\frac{1}{4}}(\log n)^5} \right) = 1 + O\left(\frac{1}{\log n} \right).\end{equation} 
 
It follows that,
\begin{align}
 &\Prob\left(e_{x}^t \Psi_t \in x + y' + [-R_h, R_h]^{2d} \times \{(i,j)\}\right) + O_A(n^{-A})\\
 & \notag \geq \left(1 + O\left(\frac{1}{\log n}\right)\right)\sum_{y_s \in S_{h-1}} \sum_k \omega_t\left(\frac{5k}{2} \right) \frac{(2R_h)^{2d}}{n^2 \left(2\pi \frac{K_n}{c_1}\right)^d}\exp\left(-\frac{\|y_s\|_2^2}{\frac{2K_h}{c}} \right)\\
 &\notag \times \Prob\left(X_{k-K_hn^2} \in x + y' + y_s' + [-R_{h-1}, R_{h-1}]^{2d} \times (\zed/n\zed)^2 \right).
\end{align}
The error in replacing $\omega_t\left(\frac{5k}{2}\right)$ with $\omega_t \left(\frac{5(k- K_hn^2)}{2} \right)$ may be removed as before.
Subject to $\|y + y_s\|_2 \leq C M$, which can be accomplished by requiring $\|y\|_2 \leq \left(1 - O\left(M^{-\frac{1}{2}}\right)\right) C M$ and  adjusting $C$ at the end, invoking the inductive assumption,
\begin{align}
 &\Prob\left(e_{x}^t \Psi_t \in x + y' + [-R_h, R_h]^{2d} \times \{(i,j)\}\right) + O_A(n^{-A})\\ \notag
 & \geq \left(1 + O\left(\frac{1}{\log n}\right)\right)\sum_{y_s \in S_{h-1}}  \frac{(2R_h)^{2d}}{n^2 \left(2\pi \frac{K_n}{c_1}\right)^d}\exp\left(-\frac{\|y_s\|_2^2}{\frac{2K_h}{c}} \right)\\
 &\notag \times c \left(1 - O\left(\frac{h-1}{\log n}\right) \right) \frac{R_{h-1}^{2d}}{M^{2d}} .
\end{align}
Since the points $y_s$ are $2R_{h-1}$ spaced, the sum over these points can be replaced with an integral, with another relative error of $\left(1 + O\left(\frac{1}{\log n}\right)\right)$.  Doing so obtains the inductive step.

 The argument is now completed by a perturbation argument which reduces from the scale $(\log n)^{21}$,  to scale $O(1)$. Throughout this argument, let $\epsilon > 0$ be a constant which can be taken arbitrarily small.
Let $w$ be a word of length $m$ beginning at $x$, ending at $y$ with the labeled pieces of $x$ and $y$ separated by at most $C M$. Let $B$ be a large fixed constant and consider the last $(\log n)^B$ steps at which piece $i$ moves.  By concentration considerations, the difference in left/right moves and up/down moves is, w.o.p. bounded by $(\log n)^{\frac{B}{2} + O(1)}$.  By Lemma \ref{empty_moves_concentration} w.o.p. the empty square moves $c_1n^2 (\log n)^B + O\left(n^2 (\log n)^{\frac{B}{2}+2}\right)$ times over this interval.  In the final $c_1 n^2 (\log n)^B + C n^2 (\log n)^{\frac{B}{2} + 2}$ moves of the empty square, w.o.p. it visits the $M^\epsilon$ neighborhood of  position of $\sP_i$ at the beginning of the interval
\begin{equation}
 O(M^{2\epsilon} (\log n)^B)
\end{equation}
times.  Let $t_1, t_2, ...$ be the traversals between the $M^\epsilon$ neighborhood of the original position of $\sP_i$ and the boundary of the $\epsilon M$ neighborhood and let $r_1, r_2, ...$ be the traversals in the opposite direction.  In each traversal $t_i, r_i$ identify a segment $t_i', r_i'$ beginning at the first point that the empty square reaches the boundary of the $\frac{\epsilon}{2} M$ neighborhood of the initial $\sP_i$ position and continuing for $\frac{\epsilon}{4}M$ moves of $\sP_e$.  With overwhelming probability, each of the segments $t_i', r_i'$ has $\frac{\epsilon}{8} M + O(M^{\frac{1}{2}}\log n)$ moves of the empty square of each type $U, D, L, R$. 

Make a right perturbation by choosing a uniform letter $L$ in the moves of $\sP_i$ at random and flipping it to $R$.  In all traversals following this change, in outgoing traversals $t_i'$ choose a uniform $R$ and flip it to $L$, while in ingoing traversals $r_i'$ choose a uniform $L$ and flip it to $R$.  Thus in the $m^\epsilon$ neighborhood of $\sP_i$, the relative position of $\sP_e$ to $\sP_i$ is unchanged from the original word, while outside the $\epsilon M$ neighborhood, the relative position of $\sP_e$ to the remaining labeled pieces is unchanged from the original word.  In particular, the event of $\sP_e$ overlapping a labeled piece is unchanged, so the probability of the perturbed word and the original word are equal.  Thus to estimate the measure of the shifted word, it suffices to estimate the relative measure of the number of unperturbed words to their image.  

In any word segment in which a letter $L$, say, is being swapped to a letter $R$, let $n_L$ and $n_R$ be the relative frequencies.  The chance that a given letter $L$ is swapped is $\frac{1}{n_L}$.  Having made a swap, the number of subwords in the preimage of the swap is $n_R + 1$, since there are $n_R + 1$ choices of letters to swap back.  This introduces a factor of $\frac{n_R+1}{n_L}$ in the relative measures.  In swapping the moves of piece $i$, this factor is $1 + O\left(\frac{1}{(\log n)^{\frac{B}{2} + O(1)}} \right)$.  In each traversal, this factor is $1 + O\left(\frac{\log n}{M^{\frac{1}{2}}} \right)$.  Taking account of the $O(M^{2\epsilon} (\log n)^B)$ traversals obtains a relative error in the measures of $1 + O\left(\frac{1}{(\log n)^{\frac{B}{2} + O(1)}} + \frac{(\log n)^{B+1}}{M^{\frac{1}{2}-2\epsilon}} \right).$ Iterating this argument $O(\log n)^{21}$ times swapping different combinations of letters proves that the measure is uniform down to scale $O(1)$.  To obtain the claim for every point rather than at scale $O(1)$, swap a hold at piece $i$ for an $R$ instead.

\end{proof}

\begin{proof}[Proof of Theorem \ref{d_pieces_theorem}]
 To prove that the $d_2$ mixing time of $(\sX_d, P_d)$ is $O(n^4)$, by Lemma \ref{d2_lemma} it suffices to prove that for each $\epsilon > 0$ there is $c > 0$ such that, as $n \to \infty$, 
 \begin{equation}
 \limsup_{n \to \infty} \sum_{1 \neq \lambda \in \sigma(P_d)} |\lambda|^{cn^4} <\epsilon.
 \end{equation}
 Note that, as $P_d$ is $\frac{1}{5}$-lazy, $\lambda \geq -\frac{4}{5}$.  Since for fixed $d$ the space grows only polynomially in $n$, it suffices to restrict attention to the part of the spectrum satisfying $\lambda \geq 1- O\left(\frac{\log n}{n^4} \right)$.
 
 By Lemma \ref{first_comparison_lemma} it suffices to show instead that for the symmetrized chain $(\sX_{d,s}, P_{d,s})$, for each $\epsilon > 0$ there is a $c>0$ such that, as $n \to \infty$,
 \begin{equation}
  \limsup_{n \to \infty} \sum_{1 \neq \lambda \in \sigma(P_{d,s})} |\lambda|^{cn^4} < \epsilon,
 \end{equation}
since the $i$th eigenvalue $\lambda_i$ of $P_d$ is bounded in terms of the $i$th eigenvalue $\lambda_{i,s}$ of $P_{d,s}$ by  $1-\lambda_i \gg 1- \lambda_{i,s}$. Again we may restrict attention to eigenvalues of $P_{d,s}$ which satisfy $\lambda \geq 1- O\left(\frac{\log n}{n^4}\right)$.

 Let $f$ be a unit eigenfunction of $P_{d,s}$ with eigenvalue $\lambda$ and write $f = f_{\ws} + f_{\nws}$ as a component which is $M = \exp(\sqrt{\log n})$ well-spaced, plus a component which is not-well-spaced.  We first reduce to considering eigenfunctions $f$ for which $\|f_{\ws}\|_2 \gg 1$.  
 
 Suppose $\|f_{\nws}\|_2 \geq 1-\epsilon$. 
  By Lemma \ref{not_well_spaced_lemma}, with $t = M^{5}n^2$ 
 \begin{align}
  \left \langle (I-\Phi_t) f, f\right \rangle &= \left \langle (I-\Phi_t) f_{\nws}, f_{\nws} \right \rangle + O\left(\sqrt{\epsilon}\right)\\ \notag
  &\geq \frac{1}{2} + O\left(\sqrt{\epsilon}\right).
 \end{align}
It follows that, if $\epsilon$ is sufficiently small, $\hat{\phi}(t(1-\lambda))$ is bounded away from 1, and hence $(1-\lambda) \gg \frac{1}{t}$. Since $\frac{1}{t}$ grows compared to $\frac{\log n}{n^4}$, this part of the spectrum may be ignored.

Now suppose $\|f_{\ws}\|_2 \gg 1$ and set $t = M^2n^2$.  Bound
\begin{align}
 \sE_{\Phi_t}\left(f, f\right) &=\frac{1}{2|\sX_{d,s}|} \sum_{x, y} (f(x)-f(y))^2 \Phi_t(x, y)\\ \notag
 &\geq \frac{1}{2|\sX_{d,s}|}\sum_{x, y \ws} (f(x)-f(y))^2 \Phi_t(x, y).
\end{align}
Call this quadratic form $\sE_{\Phi_t, \ws}$.

Let $P_{d,CM}$ be the symmetrized Markov chain introduced in Section \ref{list_of_chains} with  state space $((\zed/n\zed)^2)^{d+1}$ with $d$ labeled pieces and one empty square.  The chain holds with probability $\frac{1}{2}$, and otherwise, independently, the empty square maps to a uniform random point, while each labeled point maps to a uniform point in its $\ell^\infty$ distance $CM$ neighborhood.  Embed the well-spaced points of $\sX_{d,s}$ in $((\zed/n\zed)^2)^{d+1}$ and extend $f_{\ws}$ to a function $f$ on $((\zed/n\zed)^2)^{d+1}$ by letting the value of $f$ at a not-well-spaced point $x$ be the average value over all well-spaced points whose $\ell^\infty$ distance from $x$ is at most $CM$.  Let $\sE_{d,CM}'(f_{ws}, f_{ws}) = \sE_{d, CM}(f,f)$ be the value at the extension.  

When $y$ is not-well-spaced, let $f(y) = \sum_{x \ws} c_{x, y}f(x)$.  By Cauchy-Schwarz, for any $z$ such that $P_{d, CM}(y, z) \neq 0$, 
\begin{align}
 (f(y)-f(z))^2 &= \left(\sum_{x \ws} c_{x,y} (f(x)-f(z))\right)^2\\ \notag
 & \leq \sum_{x \ws} c(x, y) (f(x)-f(z))^2.
\end{align}
This argument may be repeated again with $z$ replacing $y$ if $z$ is not-well-spaced.  Notice that after making these replacements, the points appearing in the differences have labeled points at $\ell^\infty$ distance at most $3M$.
Since the fraction of $M$-not-well-spaced points to $M$-well-spaced points in a rectangle of sidelength $CM$ tends to 0 as $C \to \infty$, it follows that, as $C \to \infty$, $\sum_{y} c_{x, y} \to 0$.  Hence it follows that 
\begin{equation}
 \sE_{d, CM}'(f_{\ws}, f_{\ws}) \ll \frac{1}{|\sX_{d, CM}|}\sum_{x, y \ws} (f(x)-f(y))^2 P_{d, 3CM}(x, y).
\end{equation}
Since $\frac{|\sX_{d, CM}|}{|\sX_{d,s}|} \asymp 1$ this factor may be ignored.  By Lemma \ref{well_spaced_lemma}, 
\begin{equation}
 \sum_{x, y \ws} (f(x)-f(y))^2 P_{d, 3CM}(x, y) \ll \sum_{x, y \ws} (f(x)-f(y))^2 \Phi_{M^2n^2}(x, y).
\end{equation}

  By minimax, 
\begin{align}
& 1-\lambda_{i,d}  = \max\left\{\min \frac{\sE_{P_{d,CM}}(f,f)}{\|f\|_2^2}, f \in W\setminus\{0\}: \dim(W^\perp) = i\right\}\\
 \notag & \ll \max\left\{\min \frac{\sE_{P_{d,CM}}'(f,f)}{\|f\|_2^2}, f \in W\setminus\{0\}, f \; \ws : \dim(W^\perp) = i \right\}\\
 & \notag \ll \max\left\{\min \frac{\sE_{\Phi_t, \ws}(f_{\ws},f_{\ws})}{\|f_{\ws}\|_2^2}, f \in W \setminus\{0\}: \dim(W^\perp) = i\right\}\\
 \notag & \leq \max\left\{\min \frac{\sE_{\Phi_{M^2n^2}}(f,f)}{\|f_{\ws}\|_2^2}, f \in W\setminus\{0\}: \dim(W^\perp) = i\right\}.
\end{align}
Subject to the further constraint that $\lambda_{i,s} \geq 1- O\left(\frac{1}{M^5n^2}\right)$ we may assume that $\|f_{i,\ws}\|_{2}^2 \gg 1$, so that, using that the $i$th eigenvalue of $\Phi_t$ is $\hat{\phi}(M^2n^2(1-\lambda_{i,s}))$,
\begin{equation}1-\lambda_{i,d} \ll 1-\hat{\phi}(M^2n^2 (1-\lambda_{i,s})),
\end{equation}
 or $1-\lambda_{i,s} \gg \frac{1-\lambda_{i,d}}{n^2M^2}$.  

By Lemma \ref{spectral_sum_d_M}, uniformly in $n$, as $c \to \infty$, 
\begin{equation}
 \sum_{1 \neq \lambda \in \sigma(P_{d,CM})} |\lambda|^{c \frac{n^2}{C^2M^2}} \to 0
\end{equation}
as $c \to \infty$.  Thus 
\begin{equation}
 \sum_{1 \neq \lambda \in \sigma(P_{d,s})} |\lambda|^{cn^4} \to 0
\end{equation}
as $c \to \infty$, uniformly in $n$.

\end{proof}

\section{The mixing time upper bound, Proof of Theorem \ref{upper_bound_theorem}}

Let $G_n = S_{n^2-1} \times (\zed/n\zed)^2$ ($n$ odd) or $G_n = A_{n^2-1}\times (\zed/n\zed)^2$ ($n$ even) be the $n^2-1$ group.  Consider the symmetric set 
\begin{equation}
 S = \{R c, Lc, Uc, Dc, c: c = (c_3, \id), \text{$c_3$ a 3-cycle}\}
\end{equation}
and let $\mu_S$ be its uniform probability measure. 

\begin{lemma}
 The measure $\mu_S$ has $d_2$ mixing time on $G_n$ of order $O(n^2 \log n)$.
\end{lemma}
\begin{proof}
Since this is a symmetric random walk on a group, by Plancherel,
\begin{equation}
 \|\mu_S^{*N} - \bU_{G_n}\|_{d_2}^2 = \sum_{1 \neq \lambda \in \sigma(P_S)} \lambda^{2N},
\end{equation}
where $P_S$ is the transition kernel of the random walk.

 Let $\rho = \rho_1 \otimes \rho_2$ be an irreducible representation of $G_n$, so that $\rho_1$ is an irreducible representation of $S_n$ or $A_n$ and $\rho_2$ is a character of $(\zed/n\zed)^2$.  Thus $\dim \rho_1 \otimes \rho_2 = \dim \rho_1$.  We have
 \begin{align}
  \hat{\mu}_S(\rho_1 \otimes \rho_2) &= \E_{x \in \{R, L, U, D, \id\}} \left[\rho_1 \otimes \rho_2(x) \right]\E_{c_3 \text{ 3 cycle}}\left[\rho_1(c_3)  \right]\\ \notag
  &= \frac{\chi_{\rho_1}(c)}{d_{\rho_1}}\E_{x \in R, L, U, D, \id}[\rho_1 \otimes \rho_2(x)]  .
 \end{align}
where $\chi_{\rho_1}(c)$ is the character of $\rho_1$ at a 3-cycle, and $d_{\rho_1}$ is the dimension.

When $\rho_1 \otimes \rho_2$ is one dimensional, but not trivial, $\rho_1$ is the identity on $3$-cycles, so that in this case, 
$|\hat{\mu}_S(\rho_1 \otimes \rho_2)| \leq 1 - \frac{c}{n^2}$.  When $\rho_2$ is not the trivial representation, this follows by bounding the spectrum of simple random walk on the torus, while when $\rho_2$ is trivial, since $\rho_1$ is not, use that $\rho_1(\id) = 1$ while $\rho_1(R)=-1$.  Since there are $O(n^2)$ one dimensional representations, this part of the spectrum is mixed in $O(n^2 \log n)$ steps.

When $\rho_1$ has dimension $d_{\rho_1}>1$, $\frac{\chi_{\rho_1}(c)}{d_{\rho_1}}$ is the eigenvalue, with multiplicity $d_{\rho_1}$ of the 3-cycle walk in $A_{n^2-1}$ in this representation.  There are now $n^2$ representations having the same $\rho_1$ factor corresponding to the choices for $\rho_2$, each having their spectrum bounded in size by $\left|\frac{\chi_{\rho_1}(c)}{d_{\rho_1}} \right|$.  Since the spectral gap of the $3$-cycle walk is of order $\frac{1}{n^2}$, an arbitrary  factor of  $n^2$ in the multiplicity can be saved by increasing the constant in the mixing time of order $n^2 \log n$ for the 3-cycle walk.
\end{proof}
\begin{proof}[Proof of Theorem \ref{upper_bound_theorem}]
By Cauchy-Schwarz, the total variation distance is bounded by half the $d_2$ distance.  Also, since the $n^2-1$ puzzle is a symmetric random walk on a group, the $\frac{\epsilon}{|G|}-\ell^\infty$ mixing time is bounded by a constant times the $d_2$ mixing time.  Thus we only estimate the $d_2$ mixing time. 

We have
\begin{equation}
 \|e_{\id}^tP_{n^2-1}^N - \bU_{G_n}\|_{d_2}^2 = \sum_{1 \neq \lambda \in \sigma (P_{n^2-1})} \lambda^{2N}.
\end{equation}
Let $1 = \lambda_{0, n^2-1} > \lambda_{1, n^2-1} \geq \cdots  $ be the eigenvalues of $P_{n^2-1}$, and let $1 = \lambda_{0, S} > \lambda_{1,S} \geq \cdots$ be the eigenvalues of the transition kernel associated to the symmetric generating set $\mu_S$ above.

 Note that the commutator $URDL$ is a 3-cycle which leaves the empty space fixed.  Any other 3-cycle may be obtained by finding a word $w$ of length $O(n)$ which shifts any 3 pieces $i, j, k$ into the positions cycled by $URDL$ and performing $w^{-1}URDL w$, which again leaves the empty square fixed, and cycles $i,j,k$.  It follows that each element of $S$ can be obtained as a word in $O(n)$ letters on generators, so $A$ in Theorem \ref{group_comparison_theorem} may be taken $O(n^2)$. In particular, $1-\lambda_{i, n^2-1} \gg \frac{1}{n^2} (1-\lambda_{i, S})$.
 
 As $P_{n^2-1}$ is  $\frac{1}{5}$-lazy, the negative eigenvalues are bounded below by $-\frac{4}{5}$, and since the purported mixing time $O(n^4 \log n)$ is large compared to $\log|G_n| = O(n^2 \log n)$, the negative eigenvalues may be ignored when bounding the $d_2$ mixing time.  Thus by comparison, the $d_2$ mixing time is bounded by a constant times $A$ times the $d_2$ mixing time for $S$, and hence is $O(n^4 \log n)$, as wanted.
\end{proof}

\appendix

\section{Concentration inequalities}\label{concentration_appendix}
We use several exponential moment estimates which are variants of Chernoff's inequality, see \cite{TV06}.
\begin{lemma}[Chernoff's inequality]
 Let $X_1, X_2, ..., X_n$ be i.i.d. random variables satisfying $|X_i - \E[X_i]| \leq 1$ for all $i$.  Set $X:= X_1 + \cdots + X_n$ and let $\sigma := \sqrt{\Var(X)}$. For any $\lambda >0$,
 \begin{equation}
  \Prob\left(X-\E[X] \geq \lambda \sigma \right) \leq \max\left(e^{-\frac{\lambda^2}{4}}, e^{\frac{-\lambda \sigma}{2}} \right).
 \end{equation}

\end{lemma}
 The following variant handles random variables which have exponentially decaying tails. 
\begin{lemma}\label{Chernoff_variant_lemma}
 Let $X_1, X_2, ..., X_n$ be i.i.d. non-negative random variables of variance $\sigma^2$, $\sigma >0$, satisfying the tail bound, for some $c>0$ and for all $Z>0$, $\Prob(X_1 > Z) \ll e^{-cZ}$. Let $X = X_1 + X_2 + \cdots + X_n$.  Then for any $\lambda > 1$, for $c_1 = \frac{\sqrt{c\sigma}}{2}$,
 \begin{equation}
  \Prob\left(|X - \E[X]| \geq \lambda \sigma \sqrt{n}\right) \ll e^{-\frac{\lambda^2}{16}} + n e^{-c_1 \lambda^{\frac{1}{2}} n^{\frac{1}{4}}}.
 \end{equation}

\end{lemma}

\begin{proof}
 Let $Z$ be a parameter, $Z \gg n^{\frac{1}{4}}$.  Let $X_i'$ be $X_i$ conditioned on $X_i \leq Z$.  Let $\mu' = \E[X_i']$.  Let $X_i'' = X_i \cdot \one(X_i \leq Z) + \mu' \cdot \one(X_i >Z)$ and $X'' = X_1'' + X_2'' + \cdots + X_n''$. We have
 \begin{align}
  \E[X_i \cdot \one(X_i \geq Z)] &= - \int_Z^\infty x d\Prob(X_i \geq x)\\ \notag
  &=Z \Prob(X_i \geq Z) + \int_Z^\infty \Prob(X_i \geq x) dx\\ \notag
  &\ll Z e^{-cZ} + \int_Z^\infty e^{-cx}dx \leq \left(Z + \frac{1}{c}\right) e^{-cZ}.
 \end{align}
Thus, for some $c'>0$, $\E[X''] = \E[X] + O(ne^{-c'Z})$.  Also,
\begin{align}
 \Var(X_i) &= \E[(X_i-\E[X_i])^2] \\ \notag &\geq \E[(X_i-\E[X_i])^2 \one(X_i \leq Z)] \\ \notag &\geq \E[(X_i-\mu')^2 \one(X_i \leq Z)]\\& \notag = \Var(X_i'').
\end{align}
Since $|X_i''| \leq Z$, for all $n$ sufficiently large, applying Chernoff's inequality,
\begin{align}
 \Prob(|X-\E[X]| > \lambda \sigma \sqrt{n})&\leq \sum_{i=1}^n \Prob(X_i'' \neq X_i) \\ \notag &+ \Prob\left(|X'' -\E[X'']| > \frac{\lambda}{2} \sigma \sqrt{n}\right)\\ \notag
 &\ll n e^{-cZ} + 2\max\left(e^{-\frac{\lambda^2}{16}}, e^{-\frac{\lambda \sigma \sqrt{n}}{4 Z}}\right).
\end{align}
To optimize the exponents, choose $Z^2 = \frac{\lambda \sigma \sqrt{n}}{4c}$ to obtain the claim.
\end{proof}

The following variant of Chernoff's inequality gives a sharper tail bound.
\begin{lemma}
Let $p \in [0,1]$. Let $X_1, ...,X_n$ be i.i.d. with $\Prob(X_i = 1-p) = p$, $\Prob(X_i = -p) = 1-p$.  Let $X = X_1 + \cdots + X_n$. Then for any $\lambda>0$,
\begin{equation}
 \Prob(|X| > \lambda) < 2 e^{-2\lambda^2/n}.
\end{equation}

\end{lemma}
The maximum of a random walk can be controlled via the reflection principle, see e.g. \cite{TV06}.
\begin{lemma}\label{sup_bound_lemma}
 Let $\epsilon = \pm 1$ with equal probability and let $\epsilon_1, ..., \epsilon_n$ be independent samples of $\epsilon$.  Then, for any $\lambda>0$,
 \begin{equation}
  \Prob\left(\max_j \sum_{i=1}^j \epsilon_i > \lambda \right) = 2 \Prob\left( \sum_{i=1}^n \epsilon_i > \lambda\right).
 \end{equation}

\end{lemma}

Combining the previous two lemmas, it follows that
\begin{equation}
 \Prob\left(\max_j \sum_{i=1}^j \epsilon_i > \lambda\right) \leq 2 e^{-\lambda^2/(2n)}.
\end{equation}

Recall that the time $t$ heat kernel associated to a Markov matrix $P$ is $H_t = e^{-t}\sum_{k=1}^\infty \frac{t^k P^k}{k!}$.
The heat kernel is concentrated near $k = t$, as the following lemma shows.
\begin{lemma}\label{heat_kernel_concentration_lemma}
 Given $\epsilon >-1$ let \begin{equation}f(\epsilon) = (1+\epsilon)\log(1+\epsilon) -\epsilon.\end{equation}
 Then for $0 < \epsilon < 1$,
 \begin{equation}
  e^{-t} \sum_{k \leq (1-\epsilon) t} \frac{t^k}{k!} \leq \exp(-f(-\epsilon)t),
 \end{equation}
and for $\epsilon  > 0$,
\begin{equation}
 e^{-t}\sum_{k \geq (1+\epsilon)t} \frac{t^k}{k!} \leq \exp(-f(\epsilon)t).
\end{equation}

\end{lemma}
\begin{proof}
 To estimate the first sum, note that
 \begin{equation}
  1 = e^{-t(1-\epsilon)} \sum_{k } \frac{(1-\epsilon)^kt^k}{k!} \geq e^{t\epsilon} (1-\epsilon)^{(1-\epsilon)t} e^{-t}\sum_{k \leq (1-\epsilon)t} \frac{t^k}{k!}.
 \end{equation}
Thus the first sum is bounded by $\exp(-f(-\epsilon)t)$.  To estimate the second sum, bound
\begin{equation}
 1 = e^{-t(1+\epsilon)} \sum_k \frac{(1+\epsilon)^k t^k}{k!} \geq e^{-t\epsilon} (1+\epsilon)^{(1+\epsilon)t} e^{-t}\sum_{k \geq (1+\epsilon)t}\frac{t^k}{k!}.
\end{equation}
Thus the second sum is bounded by $\exp(-f(\epsilon)t)$.
\end{proof}
Since $f(\epsilon) = \frac{\epsilon^2}{2} + O(\epsilon^3)$ as $\epsilon \to 0$, the previous lemma shows that the heat kernel is concentrated about $k = t + O(\sqrt{t})$. 
The following more precise estimate follows from Stirling's approximation.
\begin{lemma}\label{stirling_approx_lemma}
 As $k \to \infty$, for $|t-k|< k^{\frac{2}{3}}$, 
 \begin{equation}
  e^{-t} \frac{t^k}{k!} = \frac{e^{-\frac{(t-k)^2}{2k}}}{\sqrt{2\pi k}} \left(1 + O\left(\frac{1}{k} + \frac{|t-k|^3}{k^2}\right)\right).
 \end{equation}

\end{lemma}
\begin{proof}
 Stirling's approximation gives \begin{equation}k! = \exp\left(\left(k + \frac{1}{2} \right)\log k -k + \frac{1}{2}\log(2\pi) + O\left(\frac{1}{k}\right) \right).\end{equation}
 Hence
 \begin{align}
 e^{-t}\frac{t^k}{k!} = \frac{1 + O\left(\frac{1}{k} \right)}{\sqrt{2\pi k}} \exp\left(k-t + k\log \frac{t}{k} \right). 
 \end{align}
Expand \begin{align}\log \frac{t}{k} &= \log \left(1 + \frac{t-k}{k}\right) \\\notag&= \frac{t-k}{k} - \frac{(t-k)^2}{2k^2} + O\left( \frac{|t-k|^3}{k^3}\right).
       \end{align}
Inserted in the exponential, this completes the estimate.
\end{proof}

Let $\phi$ be a smooth bump function, $\phi \geq 0$, $\int \phi =1$, $\supp(\phi) \subset [1,2]$, and define
$
 \Phi_t = \int_1^2 \phi(s) H_{ts}ds.
$ Write
\begin{equation}
 \Phi_t = \sum_k \omega_t\left(k\right)P^k.
\end{equation}

\begin{lemma*}[Lemma \ref{functional_concentration}]
 The function $\omega_t(k)$ satisfies $\omega_t(k) \ll \frac{1}{t}$ and the first discrete derivative satisfies
 \begin{equation}
  | \omega_t(k+1)-\omega_t(k)| \ll \frac{1}{t^2}.
 \end{equation}
For each fixed $\epsilon > 0$,  for some $c>0$,
\begin{equation}
 1-\sum_{(1-\epsilon)t \leq k \leq (2+\epsilon)t} \omega_t(k) \ll e^{-c \epsilon t}.
\end{equation}

\end{lemma*}
\begin{proof}
 We have
 \begin{equation}
  \omega_t(k) = \int_1^2 \phi(s) \frac{e^{-st}(st)^k}{\Gamma(k+1)} ds.
 \end{equation}
 To prove the first bound, let $I_k = \left[-\frac{\sqrt{k}\log t}{t}, \frac{\sqrt{k}\log t}{t} \right]$ and write this integral as
 \begin{align}
  &\int_{-\infty}^\infty \phi\left(s_0 + \frac{k}{t} \right)\frac{e^{-s_0t -k}(s_0t +k)^k}{k!}ds_0\\
  &\notag =\int_{I_k} \phi\left(s_0 + \frac{k}{t} \right)(1 + o(1)) \frac{e^{-\frac{(s_0t)^2}{2k}}}{\sqrt{2\pi k}}ds_0 + \int_{I_k^c}\phi\left(s_0 + \frac{k}{t} \right) O_A(t^{-A})ds_0 
 \end{align}
where, on $I_k$ we insert the asymptotic formula for Lemma \ref{stirling_approx_lemma} while on $I_k^c$ we insert the bound obtained from Lemma \ref{heat_kernel_concentration_lemma}. Since $\|\phi\|_1 = O(1)$ the second integral is an error term, while the first satisfies the claimed bound.
 
 Differentiating in $k$,
\begin{equation}
 \frac{\partial}{\partial k} \omega_t(k) = \int_1^2 \phi(s) \frac{e^{-st}(st)^k}{\Gamma(k+1)} \left[ \ln st - \frac{\Gamma'}{\Gamma}(k+1)\right]ds.
\end{equation}
By the asymptotic formula $\frac{\Gamma'}{\Gamma}(k+1) = \log k + O\left(\frac{1}{k}\right).$ Hence
\begin{align}
 \frac{\partial}{\partial k} \omega_t(k) &=\int_1^2 \phi(s) \frac{e^{-st}(st)^k}{\Gamma(k+1)} \left[ \ln\left(1 + \frac{st-k}{k} \right) + O\left(\frac{1}{k}\right)\right]ds.
\end{align}
If $k$ is not of order $t$, then the bound is trivially satisfied by the concentration of the heat kernel in Lemma \ref{stirling_approx_lemma}.  Otherwise, the error term $O\left(\frac{1}{k}\right)$ satisfies the claimed bound.  
To obtain the bound for the main term, as before write $s = s_0 + \frac{k}{t}$, and Taylor expand $\phi$ and the logarithm about this point to obtain an error of $O_A(t^{-A})$ from outside $I_k$, plus
\begin{align} &\int_{I_k}\phi\left(s_0 + \frac{k}{t} \right)\left(1 + O\left(\frac{1}{k} + \frac{|s_0 t|^3}{k^2}\right)\right) \frac{e^{-\frac{(s_0t)^2}{2k}}}{\sqrt{2\pi k}} \ln\left(1 + \frac{s_0 t}{k} \right) ds_0\\
 \notag &=  \int_{I_k} \left(\phi\left(\frac{k}{t}\right) + O(|s_0|)\right)\left(1 + O\left(\frac{1}{k} + \frac{|s_0 t|^3}{k^2}\right)\right) \frac{e^{-\frac{(s_0t)^2}{2k}}}{\sqrt{2\pi k}} \\&\notag\times \left( \frac{s_0 t}{k} + O\left( \frac{(s_0 t)^2}{k^2} \right)\right) ds_0.
\end{align}
Taking the main term in each position leaves an integral which vanishes, since the integrand is odd.  Since the Gaussian decays at scales $|s_0| \gg \sqrt{t}$, picking any single of the error terms results in an integral that is $O\left(\frac{1}{t^2} \right)$.

To prove the last claim of the lemma, since $\sum_k \omega_t(k) = 1$, it suffices to check that $\sum_{k < (1-\epsilon)t} \omega_t(k) + \sum_{k > (2 + \epsilon)t} \omega_t(k) \ll e^{-c\epsilon t}$.  To check this, write, for some $c>0$,
\begin{align}
& \left\{\sum_{k < (1-\epsilon)t}  + \sum_{k > (2 + \epsilon)t}\right\} \int_{1}^2 \phi(s) \frac{e^{-st}(st)^k}{k!} ds\\
\notag & = \int_1^2 \phi(s) \left\{\sum_{k < (1-\epsilon)t}  + \sum_{k > (2 + \epsilon)t}\right\} \frac{e^{-st}(st)^k}{k!} ds\\
\notag & \leq \int_1^2 \phi(s) e^{-c\epsilon st}ds 
\end{align}
by inserting the bounds of (\ref{heat_kernel_concentration_lemma}) in the integrand.  The claim now follows on adjusting $c$.
\end{proof}

\bibliographystyle{plain}

\end{document}